\documentclass[a4paper,10pt]{amsart}
\usepackage{amsmath,amsthm,amssymb,amsfonts,enumerate,color}
\usepackage{esint}
\newcommand{\R}{\mathbb{R}}
\newcommand{\N}{\mathbb{N}}
\newcommand{\tr}{\operatorname{tr}}

\newcommand{\Dom}{\operatorname{Dom}}
\newcommand{\C}{\mathcal{C}}
\newcommand{\e}{\varepsilon}

\newtheorem{theorem}{Theorem}[section]
\newtheorem{lemma}[theorem]{Lemma}
\newtheorem{corollary}[theorem]{Corollary}
\newtheorem{proposition}[theorem]{Proposition}
\newtheorem{remark}[theorem]{Remark}
\newtheorem{definition}[theorem]{Definition}

\numberwithin{equation}{section}

\author[P. R. Stinga]{Pablo Ra\'ul Stinga}
\address{Departament of Mathematics\\
          The University of Texas at Austin\\
          1 University Station, C1200\\
          Austin, TX 78712-1202\\
          United States of America}
\email{stinga@math.utexas.edu}

\author[B. Volzone]{Bruno Volzone}
\address{Dipartimento di Ingegneria\\
          Universit\`a degli Studi di Napoli ``Parthenope''\\
          Napoli, 80143\\
          Italia}
\email{bruno.volzone@uniparthenope.it}



\begin{document}

%%%%%%%%%%%%%%%%%%%%%%%%%%%%%%%%%%%%%%%%%%%%%%
\title[Fractional Neumann semilinear problems]{Fractional semilinear Neumann problems \\
arising from a fractional Keller--Segel model}
%%%%%%%%%%%%%%%%%%%%%%%%%%%%%%%%%%%%%%%%%%%%%%

%%%%%%%%%%%%%%%%%%%%%%%%%%%%%%%%%%%%%%%%%%%%%%
\begin{abstract}
We consider the following fractional semilinear Neumann problem
on a smooth bounded domain $\Omega\subset\R^n$, $n\geq2$,
$$\begin{cases}
(-\e\Delta)^{1/2}u+u=u^{p},&\hbox{in}~\Omega,\\
\partial_\nu u=0,&\hbox{on}~\partial\Omega,\\
u>0,&\hbox{in}~\Omega,
\end{cases}$$
where $\e>0$ and $1<p<(n+1)/(n-1)$.
This is the fractional version of the semilinear Neumann problem studied by Lin--Ni--Takagi
in the late 80's. The problem arises by considering steady states of the Keller--Segel model with nonlocal
chemical concentration diffusion.
Using the semigroup language for the extension method
and variational techniques,
we prove existence of nonconstant smooth solutions
{for small $\e$}, which are obtained by minimizing a suitable energy functional.
In the case of large $\e$  we obtain nonexistence of {nonconstant} solutions.
It is also shown that as $\e\to0$ the solutions $u_\e$ tend to zero in measure on $\Omega$,
while they form spikes in $\overline{\Omega}$.  The regularity estimates
of the fractional Neumann Laplacian that we develop here are essential for the analysis.
The latter results are of independent interest.
\end{abstract}
%%%%%%%%%%%%%%%%%%%%%%%%%%%%%%%%%%%%%%%%%%%%%%

%%%%%%%%%%%%%%%%%%%%%%%%%%%%%%%%%%%%%%%%%%%%%%
\maketitle
%%%%%%%%%%%%%%%%%%%%%%%%%%%%%%%%%%%%%%%%%%%%%%

%%%%%%%%%%%%%%%%%%%%%%%%%%%%%%%%%%%%%%%%%%%%%%%%%%%%%%
\section{Introduction}
%%%%%%%%%%%%%%%%%%%%%%%%%%%%%%%%%%%%%%%%%%%%%%%%%%%%%%

In the famous paper \cite{Lin-Ni-Takagi} C.-S. Lin, W.-M. Ni and I. Takagi
studied the existence of solutions to the one-parameter semilinear Neumann boundary value problem
\begin{equation}\label{LinTakagiprobl}
\begin{cases}
-\e\Delta u+u=g(u),&\hbox{in}~\Omega,\\
\partial_\nu u=0,&\hbox{on}~\partial\Omega.
\end{cases}
\end{equation}
Here $\Omega$ is {a smooth} bounded domain of $\R^{n}$, $n\geq1$, $\e$ is a positive parameter,
$\nu$ is the outer unit normal to $\partial\Omega$ and $g$ is a suitable nonnegative nonlinearity on $\R$
vanishing for $t\leq0$, {growing superlinearly at infinity} and such that, among other structural properties, $g(t)=O(t^{p})$ as $t\rightarrow+\infty$, where $p>1$ is below the \emph{critical} Sobolev exponent $(n+2)/(n-2)$ when $n\geq3$. In particular, the authors study the existence of positive weak solutions to \eqref{LinTakagiprobl} by
applying the Mountain Pass Lemma of A. Ambrosetti and P. Rabinowitz \cite{AMBRRAB} to the energy functional
\[
\mathcal{J}_{\e}(u)=\frac{1}{2}\int_{\Omega}\left(
\e|\nabla u|^{2}+u^2\right)dx-\int_{\Omega}G(u)\,dx,
\]
where $G$ is an antiderivative of $g$.
It is proved in \cite[Theorem~2]{Lin-Ni-Takagi} that if $\e$ is small enough, there exists a positive smooth solution $u_{\e}$ (a \emph{critical least energy solution}) for which
\[
\mathcal{J}_{\e}(u_\e)\leq C\e^{n/2},
\]
for a positive constant $C$ which does not depend on $\e$. This property allows
to prove that the family of solutions $\left\{u_{\e}\right\}_{\e>0}$ is uniformly bounded for small $\e$
and to obtain the convergence of such solutions to 0 in measure when $\e\rightarrow0$,
see \cite[Corollary 2.1]{Lin-Ni-Takagi}. Actually, for the nonlinearity $g(t)=t^p$ in \cite{Lin-Ni-Takagi} it is shown that the boundedness of $\left\{u_{\e}\right\}_{\e>0}$ holds for all $\e>0$ and that $u\equiv1$ is the only positive solution to problem \eqref{LinTakagiprobl} for large $\e$. Besides, in \cite[Proposition~4.1]{Lin-Ni-Takagi} the authors exhibit a striking property concerning the shape of the graphs of the solutions $\left\{u_{\e}\right\}_{\e>0}$ for small $\e$, which actually look like spikes in $\overline{\Omega}$. Such a property was the starting point of the research continued in the papers \cite{NiTakagi2,NiTakagi3} concerning mainly the localization of these spikes on the boundary of $\Omega$.

The aim of the present paper is to extend all the previous cited results to the problem
\begin{equation} \label{eq.8}
\begin{cases}
(-\e\Delta)^{1/2}u+u=g(u),&\hbox{in}~\Omega,\\
\partial_\nu u=0,&\hbox{on}~\partial\Omega,
\end{cases}
\end{equation}
where $\Omega$ is as before and $\e>0$.  The nonlinearity $g$ is defined as
\begin{equation}\label{nonlinearity}
g(t)=\begin{cases}
t^{p},&\hbox{if}~t\geq0,\\
0,&\hbox{if}~t\leq0.
\end{cases}
\end{equation}
for
\begin{equation}\label{eq.9}
1<p<\frac{n+1}{n-1}.
\end{equation}
Notice that $(n+1)/(n-1)$ is the critical Sobolev \emph{trace} exponent.
The operator $(-\e\Delta)^{1/2}$ is understood as the square root of the Laplacian
in the bounded domain $\Omega$ encoding the homogeneous Neumann boundary condition,
that is, the \emph{fractional Neumann Laplacian} which is defined as follows.
Let $\left\{\varphi_k\right\}_{k\in\N_0}$ and $\left\{\lambda_{k}\right\}_{k\in\N_0}$
be the eigenfunctions and eigenvalues of the Laplacian $-\Delta$ on $\Omega$
with homogeneous Neumann {boundary} condition. The ($\e$--)Neumann Laplacian is the operator
that acts on an $L^2$ function $u(x)=\sum_{k=0}^\infty u_k\varphi_k(x)$ as
$$-\e\Delta_Nu(x)=\sum_{k=0}^\infty (\e\lambda_k)u_k\varphi_k(x),$$
in a suitable sense. Then the fractional ($\e$--)Neumann Laplacian is given by
$$(-\e\Delta_N)^{1/2}u(x)=\sum_{k=0}^\infty(\e\lambda_k)^{1/2}u_k\varphi_k(x).$$
For more details see Sections \ref{Section1} and \ref{Section:lineal}.

Using the language of semigroups
as introduced in \cite{Stinga}, see also \cite{Stinga-Torrea}, one can check that $(-\e\Delta_N)^{1/2}$
is indeed a nonlocal operator. In fact,
$$(-\e\Delta_N)^{1/2}u(x)=\frac{1}{\sqrt{\pi}}\int_0^\infty\big(e^{t\e\Delta_N}u(x)
-u(x)\big)\,\frac{dt}{t^{3/2}},\quad x\in\Omega,$$
where $e^{t\Delta_N}u(x)$ is the heat diffusion semigroup generated by the Neumann Laplacian acting on $u$.
Then, by following the ideas of \cite{Stinga,Stinga-Torrea},
one can conclude that for a smooth function $u$ we have the pointwise integro-differential formula
$$(-\e\Delta_N)^{1/2}u(x)=c_{n,\Omega}
\operatorname{P.V.}\int_\Omega\big(u(x)-u(z))K(x,z)\,dz,\quad x\in\Omega,$$
where $c_{n,\Omega}$ is a positive constant and the kernel $K(x,z)$, given in terms of the heat kernel
for the Neumann Laplacian, satisfies the estimate $K(x,z)\sim\e^{1/2}|x-z|^{-(n+1)}$, for $x,z\in\Omega$.
Moreover, the fundamental solution of $(-\e\Delta_N)^{1/2}$ behaves like $\e^{1/2}|x-z|^{-(n-1)}$. We will
not go further into these details here.

Looking at problem \eqref{eq.8} and taking into account the nonlocality of the fractional Neumann Laplacian, a considerable issue that arises is to provide a suitable definition of weak solution.
The answer relies in understanding $(-\e\Delta_N)^{1/2}u$ as the normal derivative on $\Omega$
of a proper harmonic extension of $u$ to the cylinder ${\mathcal{C}:=}\Omega\times(0,\infty)$.
Let us first explain this idea in the classical case.
It is known that if $v(x,y):\R^n\times(0,\infty)\to\R$ is the harmonic extension of a smooth function $u:\R^n\to\R$
to the upper half space, namely, if $v$ is the solution to
$$\begin{cases}
\Delta_xv+v_{yy}=0,&\hbox{in}~\mathbb{R}^n\times(0,\infty),\\
v(x,0)=u(x),&\hbox{on}~\mathbb{R}^n,
\end{cases}$$
then
$$-v_y(x,0)=(-\Delta)^{1/2}u(x),\quad x\in\mathbb{R}^n,$$
where $(-\Delta)^{1/2}$ is the square root of the Laplacian on $\mathbb{R}^n$.
Such an identity can be checked by using the Fourier transform on the variable $x$.
We can understand this last property by using the language of the semigroups as explained in \cite{Stinga, Stinga-Torrea}.
Indeed, the solution $v$ above is the \textit{Poisson semigroup} generated by the Laplacian on $\R^n$,
which can be written in terms of the classical heat semigroup $e^{t\Delta}$
 via the so called \textit{Bochner subordination formula}:
$$v(x,y)\equiv e^{-y(-\Delta)^{1/2}}u(x)=\frac{y}{2\sqrt{\pi}}\int_0^\infty e^{-y^2/(4t)}e^{t\Delta}u(x)\,\frac{dt}{t^{3/2}}.$$
Then,
$$-v_y(x,y)=(-\Delta)^{1/2}e^{-y(-\Delta)^{1/2}}u(x),~y>0,\quad\hbox{and}\quad-v_y(x,y)\big|_{y=0}=(-\Delta)^{1/2}u(x).$$
The advantage now is that nothing prevents us to replace $-\Delta$ in the formulas
for the semigroup above by a positive operator $L$ (defined together with its functional
domain in order to include boundary conditions), so the same characterization
for $L^{1/2}$ as a Dirichlet-to-Neumann operator holds, see
\cite{Stinga, Stinga-Torrea}, or \cite{Gale-Miana-Stinga} for more general operators.
Moreover, this semigroup language avoids the use of the Fourier transform, which
is not available for a general differential operator $L$.
 In our particular case,
the result for the $\e$--fractional Neumann Laplacian reads as follows.
If $v(x,y)$ is the solution to
$$\begin{cases}
\e\Delta_xv+v_{yy}=0,&\hbox{in}~\mathcal{C}:=\Omega\times(0,\infty),\\
\partial_{\nu} v=0,&\hbox{on}~\partial_L\mathcal{C}:=\partial\Omega\times[0,\infty),\\
v(x,0)=u(x),&\hbox{on}~\Omega,
\end{cases}$$
then
$$-v_y(x,0)=(-\e\Delta_N)^{1/2}u(x),\quad x\in\Omega.$$
See Section \ref{Section1} for more details.

Going back to our problem \eqref{eq.8},
we can define a weak solution $u$ as the trace over $\Omega$
of a weak solution $v$
 to the problem with semilinear Neumann boundary condition
\begin{equation}\label{semilprobintrod}
\begin{cases}
\e\Delta_xv+v_{yy}=0,&\hbox{in}~\mathcal{C},\\
\partial_{\nu} v=0,&\hbox{on}~\partial_{L}\mathcal{C},\\
-v_y(x,0)+v(x,0)=g(v(x,0)),&\hbox{on}~\Omega.
\end{cases}
\end{equation}
that is
$$u(x)=v(x,0),\quad x\in\Omega.$$
The energy functional related to \eqref{semilprobintrod} is
\begin{equation}\label{4 star}
\mathcal{I}_\e(v)=\frac{1}{2}\iint_{\mathcal{C}}\big(\e|\nabla_xv|^2+v_{y}^2\big)\,dx\,dy
{+\frac{1}{2}\int_\Omega v(x,0)^2\,dx}-
\int_\Omega G(v(x,0))\,dx,
\end{equation}
where $G$ is an antiderivative of the nonlinearity $g$ given in \eqref{nonlinearity}. {Here we can}
 apply the techniques of the Calculus of Variations
to prove existence of critical points of $\mathcal{I}_\e$, yielding
{nonconstant} regular solutions $u$ to \eqref{eq.8}.

The present paper arises from a concrete application
that can also be seen as a nonlocal version of the model where the investigation in \cite{Lin-Ni-Takagi} took its origin.
Problem \eqref{eq.8}
appears when considering the steady states of the Keller--Segel system when
the diffusion of the concentration of the chemical is nonlocal.
Indeed, while \cite{Lin-Ni-Takagi} is involved in the study of stationary solutions to the classical Keller--Segel chemotaxis model system posed in $\Omega$, we look the
solutions to \eqref{eq.8} as steady states to the following local-nonlocal system
\begin{equation}
\begin{cases}
\rho_t-D_1\Delta \rho+\chi\nabla\cdot(\rho\nabla \log c)=0,&\hbox{in}~\Omega\times(0,\infty),\\
c_t+D_2(-\Delta)^{1/2} c+ac-b\rho=0,&\hbox{in}~\Omega\times(0,\infty),\\
\partial_\nu \rho= \partial_\nu c=0,&\hbox{on}~\partial\Omega\times(0,\infty),
\end{cases}\label{nonlockelseg}
\end{equation}
for some positive constants $D_1,\,D_2,\,\chi,\,a,\,b$, with the initial conditions
\begin{equation}
\begin{cases}
\rho(x,0)=\rho_0(x),&\hbox{on}~\Omega,\\
c(x,0)=c_0(x),&\hbox{on}~\Omega.
\end{cases}\label{nonlockelseginitcond}
\end{equation}
Here $\rho(x,t)$ describes the density of a bacteria population (such as amoebae) and $c(x,t)$ is
 the density concentration of a certain chemical.
In the model \eqref{nonlockelseg}--\eqref{nonlockelseginitcond} the diffusion is assumed to be
local for the bacteria, while it is nonlocal for the chemical.
 Due to the conservation of mass property for the first equation in
 \eqref{nonlockelseg}, a steady state for the system is a couple of functions $u,v$ satisfying
\begin{equation}
\begin{cases}
D_1\Delta u-\chi\nabla\cdot(u\nabla \log v)=0,&\hbox{in}~\Omega,\\
D_2(-\Delta)^{1/2}v+av-bu=0,&\hbox{in}~\Omega,\\
\partial_\nu u = \partial_\nu v=0,&\hbox{on}~\partial\Omega,
\end{cases}\label{nonlockelstat}
\end{equation}
with the condition
\begin{equation}\label{eq:mean}
u_\Omega:=\frac{1}{|\Omega|}\int_{\Omega}u(x)\,dx
\end{equation}
equal to an assigned positive constant $\overline{u}$.
Of course $u=\overline{u}$,
$v=a^{-1}b\overline{u}$ is a solution to \eqref{nonlockelstat}. Thus it makes sense to look for
positive nonconstant solutions.
Arguing similarly to \cite{Lin-Ni-Takagi}, it is possible to write $u=\lambda v^{\chi/D_1}$
for some positive constant $\lambda$. Then the system \eqref{nonlockelstat} is equivalent
to find a solution to \eqref{eq.9}, where $\e=D_2/a$, $g(t)=t^{p}$ for $t\geq0$ and $p=\chi/D_1$ and $v_\Omega=\overline{v}$.

Since the appearance of the papers by L. Caffarelli, L. Silvestre and collaborators
\cite{Caffree, Caffarelli-Silvestre, Caffarelli-Salsa-Silvestre, Caffgeo,  Silvestre-thesis, Silv1, Silv2}
nonlocal PDEs with fractional Laplacians became a topic which is nowadays deserving a lot of attention,
{also} because of the various applications in several fields. As far as the Neumann
Laplacian is concerned, some problems were studied in \cite{Alves-Oliva, Imbert-Mellet, PellacciMontef}.
A fractional Keller--Segel
model was considered in \cite{Escudero}, though the author there proposes to model the concentration
of the chemical with the usual local diffusion while the bacteria satisfy a nonlocal diffusion in dimension one.
This is in contrast with our model.

Notice that we could also model the diffusion of the chemical with a fractional Neumann Laplacian
with power different than $1/2$. The extension problem for such an operator is available (see the
generalization of the Caffarelli--Sivestre result of \cite{Caffarelli-Silvestre} given in
\cite{Stinga, Stinga-Torrea}) so in principle part of
the analysis could be carried out. This generalization will appear elsewhere.
The nonlinearity \eqref{nonlinearity} which we use could be replaced by a more general nonlinearity $g(x,t)$
under suitable structural conditions as in \cite{Lin-Ni-Takagi} without affecting the statements of the
results or the techniques needed. We stick to \eqref{nonlinearity}
to keep a clean presentation of the ideas.

The paper is organized as follows. In Section \ref{Section1} we give the functional framework which is necessary for the analysis of \eqref{eq.8}.
In particular, by making use of the extension problem (Theorem \ref{extensth}) and the efficient semigroup
language approach, we identify the domain of the operator $(-\e\Delta_{N})^{1/2}$ (Theorem \ref{thm:trace and H Omega}) and exhibit a trace inequality (Lemma \ref{thm:traces}), valid for functions belonging to a suitable functional space on the cylinder $\mathcal{C}$. All this will serve to give sense to the definition of weak solution for the semilinear problem \eqref{semilprobintrod}. Section \ref{Section:lineal} is entirely devoted to discuss existence (Lemma \ref{existweaksollin}), Harnack estimates (Theorem \ref{Harnack}) and regularity (Theorem \ref{Thm:regularity}) of weak solutions to the linear problem
\begin{equation} \label{eq.3}
\begin{cases}
(-\e\Delta)^{1/2}u+u=f,&\hbox{in}~\Omega,\\
\partial_\nu u=0,&\hbox{on}~\partial\Omega.
\end{cases}
\end{equation}
{We point out that for the equation $(-\Delta_N)^{1/2}u=f$ parallel results can be
found in \cite{PellacciMontef}, while for the case of the fractional Dirichlet Laplacian
$(-\Delta_D)^{1/2}u=f$ one can see \cite{Cabre-Tan}.}
Section \ref{Section4} introduces the proper concept of weak solution for the semilinear problem \eqref{eq.8} and shows the existence of a positive {nonconstant} smooth solution $u_\e$ for small $\e$. {Its} extension $v_\e$ is obtained as a critical point of the energy {functional \eqref{4 star}} associated to problem \eqref{semilprobintrod} (Theorem \ref{TeoremaMountain}, Corollary \ref{Cor:u epsilon}). In Section \ref{Section5} we employ a Moser iteration argument to obtain, for small $\e>0$, an $L^{\infty}$ bound of the sequence $\left\{u_\e\right\}_{\e>0}$ (Theorem \ref{Moser}). From here an $L^{q}$ estimate of each $u_{\e}$ is derived (Lemma \ref{Lemma 4.1}). This information allows to describe the geometry of the functions $\{u_\e\}_{\e>0}$, that are shown to be spike solutions (Theorem \ref{Shape}).
In Theorem \ref{Unifboundtheo} we use the previous results and a blow-up argument to get
the boundedness of any solution $\left\{u_\e\right\}_{\e>0}$ independent of $\e$.
Finally, Theorem \ref{existconst} concludes the paper by showing the nonexistence of nonconstant positive solutions to \eqref{eq.8} for large $\e$.

Throughout the paper $C,C_0,c$ denote positive constants. Subscripts in the constants point out the dependence on a group of parameters. By $\Omega$ we denote a {smooth
bounded domain} of $\R^{n}$, $n\geq2$. For positive numbers $A,B$ the symbol $A\sim B$ means
that for positive constants $c,C$ we have $cA\leq B\leq CA$, and we call $c,C$ the equivalence constants.

%%%%%%%%%%%%%%%%%%%%%%%%%%%%%%%%%%%%%%%%%%%%%%%%%%%%%%
\section{Extension problem and domain for $(-\e\Delta_N)^{1/2}$}\label{Section1}
%%%%%%%%%%%%%%%%%%%%%%%%%%%%%%%%%%%%%%%%%%%%%%%%%%%%%%

We denote by $\langle\cdot,\cdot\rangle_{L^{2}(\Omega)}$, $\langle\cdot,\cdot\rangle_{H^{1}(\Omega)}$ the scalar products in $L^{2}(\Omega)$ and in the usual Sobolev space $H^{1}(\Omega)$, respectively, and
$\langle\cdot,\cdot\rangle$ will mean the pairing between a Hilbert space and its dual.
Consider the homogeneous Neumann eigenvalue problem for the Laplacian in $\Omega$:
\begin{equation} \label{eq.1}
\begin{cases}
 -\Delta\varphi=\lambda\varphi,&\hbox{in}~\Omega,\\
 \partial_\nu\varphi=0 ,&\hbox{on}~\partial\Omega,
\end{cases}
\end{equation}
for $\lambda\geq0$. It is well known (see for example \cite{Evans,Gilbarg-Trudinger})
 that there exists a sequence of nonnegative eigenvalues
$\{\lambda_{k}\}_{k\in\N_{0}}$ that corresponds to eigenfunctions $\{\varphi_{k}\}_{k\in\N_{0}}$ in $H^{1}(\Omega)$
which are weak solutions to \eqref{eq.1}. We have that $\lambda_{0}=0$, $\varphi_{0}=1/\sqrt{|\Omega|}$,
$\int_{\Omega}\varphi_{k}\, dx=0$, for all $k\geq1$ and each $\varphi_k$ belongs to $C^\infty(\overline{\Omega})$.
The eigenfunctions $\{\varphi_k\}_{k\in\N_0}$ form an orthonormal basis of $L^2(\Omega)$.
By using the $L^2$ normalization and the weak formulation of the equation we see that
$\|\varphi_k\|_{H^1(\Omega)}^2=1+\lambda_k$.
It is easy to check that $\{\varphi_{k}\}_{k\in\N_{0}}$ is also an orthogonal basis of $H^1(\Omega)$. Hence,
since  $\langle u,\varphi_k\rangle_{H^1(\Omega)}=(1+\lambda_k)\langle u,\varphi_k\rangle_{L^2(\Omega)}$, we find
\begin{equation}\label{eq.H1}
H^1(\Omega)=\Big\{u \in L^2(\Omega):\|u\|_{H^1(\Omega)}^2=
\sum_{k=0}^\infty(1+\lambda_k)|\langle u,\varphi_k\rangle_{L^2(\Omega)}|^2<\infty\Big\}.
\end{equation}

For a function $u\in H^1(\Omega)$, we define the (negative) Neumann Laplacian of $u$ as {an} element
$-\Delta_Nu$ in the
dual space $H^1(\Omega)'$ verifying
\begin{equation}\label{eq:estrella}
\langle-\Delta_Nu,v\rangle=\int_{\Omega}\nabla u\cdot\nabla v\,dx,\quad\hbox{for every}~v\in H^1(\Omega).
\end{equation}
In terms of the orthogonal basis $\{\varphi_k\}_{k\in\N_0}$ we can write
$$-\Delta_Nu=\sum_{k=1}^\infty\lambda_k\langle u,\varphi_k\rangle_{L^2(\Omega)}\varphi_k,
\quad\hbox{in}~H^1(\Omega)',$$
namely, for each $v\in H^1(\Omega)$ we have
\begin{equation}\label{eq.2.2}
\langle-\Delta_Nu,v\rangle=\sum_{k=1}^\infty\lambda_k\langle u,\varphi_k\rangle_{L^2(\Omega)}
\langle v,\varphi_k\rangle_{L^2(\Omega)}.
\end{equation}
Indeed, by the weak formulation of the eigenvalue problem,
${\langle -\Delta_Nu,\varphi_k\rangle}
=\lambda_k\langle u,\varphi_k\rangle_{L^2(\Omega)}$, for all $k\geq0$, and
the identity \eqref{eq.2.2} follows by linearity and density.
Notice that $-\Delta_N$ has a nontrivial kernel in $H^1(\Omega)$,
namely, the set of constant functions.

In order to define the fractional $\e$--Neumann Laplacian
$(-\e\Delta_N)^{1/2}$ for each $\e>0$, we consider the Hilbert space
$$\mathcal{H}_\e(\Omega)\equiv\Dom((-\e\Delta_N)^{1/2}):=
\Big\{u\in L^2(\Omega):\sum_{k=1}^\infty(\e\lambda_k)^{1/2}|\langle u,\varphi_k
\rangle_{L^2(\Omega)}|^2<\infty\Big\},$$
under the scalar product
\[
\langle u,v\rangle_{\mathcal{H}_{\e}(\Omega)}:=\langle u,v\rangle_{L^{2}(\Omega)}+\e^{1/2}\sum_{k=1}^{\infty}\lambda_{k}^{1/2} \langle
u,\varphi_k\rangle_{L^2(\Omega)}\langle
v,\varphi_k\rangle_{L^2(\Omega)},
\]
so that the norm in $\mathcal{H}_\e(\Omega)$ is given by
$$\|u\|^2_{\mathcal{H}_\e(\Omega)}=\|u\|_{L^2(\Omega)}^{2}+\e^{1/2}\sum_{k=1}^{\infty}\lambda_{k}^{1/2}|\langle
u,\varphi_k\rangle_{L^2(\Omega)}|^{2}.$$
Since $\lambda_k\nearrow\infty$, it is obvious that $C^\infty(\overline{\Omega})\subset H^1(\Omega)\subset\mathcal{H}_\e(\Omega)$.
For $u\in\mathcal{H}_\e(\Omega)$, we define $(-\e\Delta_N)^{1/2}u$ as {an} element in the dual space
$\mathcal{H}_\e(\Omega)'$ given by
$$(-\e\Delta_N)^{1/2}u=\sum_{k=1}^\infty(\e\lambda_k)^{1/2}\langle u,\varphi_k\rangle_{L^2(\Omega)}\varphi_k,
\quad\hbox{in}~\mathcal{H}_\e(\Omega)',$$
that is, for any function $v\in\mathcal{H}_\e(\Omega)$,
$$\langle(-\e\Delta_N)^{1/2}u,v\rangle=\sum_{k=1}^\infty(\e\lambda_k)^{1/2}\langle u,\varphi_k\rangle_{L^2(\Omega)}
\langle v,\varphi_k\rangle_{L^2(\Omega)}.$$
As before, the set of constant functions
is the nontrivial kernel {of
$(-\e\Delta_N)^{1/2}$ in $\mathcal{H}_\e(\Omega)$}.
The last identity can be rewritten, in a parallel way to \eqref{eq:estrella}, as
$$\langle(-\e\Delta_N)^{1/2}u,v\rangle=\int_\Omega(-\e\Delta_N)^{1/4}u(-\e\Delta_N)^{1/4}v\,dx,\quad\hbox{for
every}~v\in\mathcal{H}_\e(\Omega),$$
where $(-\e\Delta_{{N}})^{1/4}$ is defined by taking the power $1/4$ of the eigenvalues $\lambda_k$.

Before presenting the extension problem for $(-\e\Delta_N)^{1/2}$
we collect some well known facts about
the heat equation {on $\Omega$}
with Neumann boundary condition.
Let $u\in L^p(\Omega)$, $1\leq p\leq\infty$. Then the unique classical solution to the Neumann heat equation
\begin{equation*}
\begin{cases}
w_t=\Delta w,&\hbox{in}~{\Omega\times(0,\infty)},\\
\partial_\nu w=0,&\hbox{on}~\partial\Omega\times{[0,\infty)},\\
w(x,0)=u(x),&\hbox{on}~\Omega,
\end{cases}
\end{equation*}
is given by
$$w(x,t)\equiv e^{t\Delta_N}u(x)=\int_{\Omega}\mathcal{W}_t(x,z)u(z)\,dz,$$
where
$$\mathcal{W}_t(x,z)=\sum_{k=0}^\infty e^{-t\lambda_k}\varphi_k(x)\varphi_k(z),$$
is the (distributional) Neumann heat kernel. We have
\begin{equation}\label{integral heat kernel}
\int_{\Omega}\mathcal{W}_t(x,z)\,dz=1,\quad\hbox{for all}~x\in\Omega,~t>0.
\end{equation}
Moreover, for each $T>0$
there exist positive constants $c_1,c_2,c_3,c_4,c,C$, depending only on $\Omega$, $n$ and $T$, such that
\begin{equation}\label{Saloff-Coste}
c_1\frac{e^{-\frac{|x-z|^2}{c_2t}}}{t^{n/2}}\leq
\mathcal{W}_t(x,z)\leq c_3\frac{e^{-\frac{|x-z|^2}{c_4t}}}{t^{n/2}},\quad\hbox{for all}~x,z\in\Omega,~0<t<T,
\end{equation}
and
\begin{equation}\label{Wang-Yan}
 |\nabla_x\mathcal{W}_t(x,z)|\leq \frac{C}{t^{(n+1)/2}}e^{-c\frac{|x-z|^2}{t}},\quad\hbox{for all}~ x,z\in\Omega,~t>0.
\end{equation}
Moreover, there is a constant $M$ such that $|\mathcal{W}_{t}(x,z)|\leq M$ for all $x,z\in\Omega$
and $t>T$.
 For these properties see \cite{Gyrya-Saloff-Coste,
Saloff-Coste, Wang-Yan-Preprint, Wang-Yan-Proceedings, Wang}.

We particularize to $(-\e\Delta_N)^{1/2}$ the general extension problem proved in \cite{Stinga, Stinga-Torrea}.

\begin{theorem}[Extension problem for $(-\e\Delta_N)^{1/2}$]\label{extensth}
Let $u\in \mathcal{H}_\e(\Omega)$ such that $\displaystyle\int_\Omega u\,dx=0$. Define
\begin{equation}v(x,y)= e^{-y(-\e\Delta_N)^{1/2}}u(x):=\sum_{k=1}^{\infty}e^{-y(\e \lambda_k)^{1/2}}
\langle u,\varphi_k\rangle_{L^2(\Omega)}\varphi_{k}(x).\label{trueextens}\end{equation}
Then $v\in H^1(\mathcal{C})$ is the unique weak solution to the extension problem
\begin{equation}\label{eq.4}
\begin{cases}
 \e\Delta_{x} v+v_{yy}=0,&\hbox{in}~\mathcal{C},\\
 \partial_{\nu}v=0,&\hbox{on}~\partial_L\mathcal{C},\\
 v(x,0)=u(x),&\hbox{on}~\Omega,
\end{cases}
\end{equation}
where $\nu$ is the outward normal
to the lateral boundary $\partial_L\mathcal{C}$ of $\mathcal{C}$,
such that $\displaystyle\int_\Omega v(x,y)\,dx=0$, for all $y\geq0$.
More precisely,
$$\iint_{\mathcal{C}}\big(\e\nabla_{x}v\cdot\nabla_{x}\psi+ v_y\psi_y\big) \,dx\,dy=0,$$
for all test functions $\psi\in H^1(\mathcal{C})$ with zero trace over $\Omega$, $\tr_\Omega\psi=0$,
and $\lim_{y\to0^+}v(x,y)=u(x)$ in $L^2(\Omega)$. Furthermore, the function $v$ is the unique minimizer of the energy functional
\begin{equation}\mathcal{F}(v)=\frac{1}{2}\iint_{\mathcal{C}}\big(\e|\nabla_{x}v|^{2}+|v_{y}|^{2}\big)\,dx\,dy,\label{energufunct}\end{equation}
over the set  $\mathcal{U}=\left\{v\in H^{1}(\mathcal{C}):\,\tr_{\Omega}v=u\right\}$. We can also write
\begin{equation}\label{eq.5}
v(x,y)=\frac{\e^{1/2}y}{2\sqrt{\pi}}\int_{0}^{\infty}e^{-\e y^2/(4t)}e^{t\Delta_N}u(x)\,\frac{dt}{t^{3/2}}
=\int_{\Omega}\mathcal{P}_{\e^{1/2}y}(x,z)u(z)\,dz,
\end{equation}
where, for any $y>0$,
\begin{equation}\label{Poisson kernel}
\mathcal{P}_y(x,z)=\sum_{k=0}^\infty e^{-y\lambda_k^{1/2}}\varphi_k(x)\varphi_k(z)
=\frac{y}{2\sqrt{\pi}}\int_{0}^{\infty}e^{-y^2/(4t)}\mathcal{W}_t(x,z)\,\frac{dt}{t^{3/2}},
\end{equation}
is the Neumann--Poisson kernel. An equivalent formula for $v$ is
$$v(\cdot,y)=\frac{1}{\sqrt{\pi}}\int_0^\infty e^{-\e y^2/(4t)}e^{t\Delta_N}((-\Delta_N)^{1/2}u)\,\frac{dt}{t^{1/2}},
\quad\hbox{in}~\mathcal{H}_\e(\Omega)'.$$
Moreover,
\begin{equation}\label{eq.36}
-\lim_{y\to0^+}v_y=(-\e\Delta_N)^{1/2}u,\quad\hbox{in}~\mathcal{H}_\e(\Omega)'.\end{equation}
\end{theorem}

\begin{proof}
From \cite{Stinga,Stinga-Torrea}
we know that $v\in C^\infty((0,\infty);H^1(\Omega))\cap C([0,\infty);L^2(\Omega))$. Observe that
$$\iint_{\mathcal{C}}v^2\,dx\,dy=\sum_{k=1}^{\infty}\int_0^\infty e^{-2y(\e \lambda_k)^{1/2}}
|\langle u,\varphi_k\rangle_{L^2(\Omega)}|^2\,dy\leq C\|u\|_{L^2(\Omega)}^2,$$
and
\begin{equation}\label{eq.32}
\begin{aligned}
\iint_{\mathcal{C}}\left(\e|\nabla_xv|^2+v_y^2\right)\,dx\,dy&= 2\sum_{k=1}^{\infty}\e\lambda_k|\langle u,\varphi_k\rangle_{L^2(\Omega)}|^2\int_0^\infty e^{-2y(\e \lambda_k)^{1/2}}dy\\
&=\sum_{k=1}^\infty(\e\lambda_k)^{1/2}|\langle u,\varphi_k\rangle_{L^2(\Omega)}|^2
=\|(-\e\Delta_{{N}})^{1/4}u\|^{2}_{L^{2}(\Omega)}.
\end{aligned}
\end{equation}
Therefore $v\in H^1(\mathcal{C})$.
Let $\psi\in H^1(\mathcal{C})$ such that $\tr_\Omega\psi=0$. For almost every $y>0$ we can write
$$\psi(x,y)=\sum_{k=0}^\infty\langle
\psi(\cdot,y),\varphi\rangle_{L^2(\Omega)}\varphi_k(x)=:\sum_{k=0}^\infty\psi_k(y)\varphi_k(x).$$
Then, by using \eqref{eq.2.2} and the definition of $v$, for almost every $y>0$ we have
$$\int_\Omega v_{yy}\psi\,dx=\sum_{k=1}^\infty \e\lambda_ke^{-y(\e\lambda_k)^{1/2}}
\langle u,\varphi_k\rangle_{L^2(\Omega)}\psi_{k}(y)=\e\int_\Omega\nabla_xv\cdot\nabla_x\psi\,dx.$$
Integrating this identity in $y$ and applying integration by parts,
$$ \iint_{\mathcal{C}}\big(\e\nabla_xv\cdot\nabla_x\psi+v_y\psi_y\big)\,dx\,dy =
-\lim_{y\to0^+}\int_\Omega v_y(x,y)\psi(x,y)\,dx=0.$$
Hence $v$ in the statement is a weak solution to \eqref{eq.4}.

Uniqueness can be proved in two ways. One is by writing $v(x,y)=\sum_{k=0}^\infty v_k(y)\varphi_k(x)$
and showing that each $v_k(y)$ is the unique solution to an ordinary differential equation, see \cite{Stinga, Stinga-Torrea}.
Another way of proving uniqueness is by using the weak formulation for the difference $V$
of two solutions to \eqref{eq.6}. Indeed, using $V$ itself as a test function in the weak formulation, we get that
$\nabla_{x,y}V=0$. But then, since $V\in H^1(\mathcal{C})$,
we must have $V\equiv 0$ on $\mathcal{C}$.

Concerning the minimization property of $v$, let us take any other function $w\in\mathcal{U}$ and use
 $v-w$ as a test function in the weak formulation of \eqref{eq.4}. Then we get
\[
\mathcal{F}(v)=\frac{1}{2}\iint_{\mathcal{C}}\left(\e\nabla_{x}v\cdot\nabla_{x}w+v_{y}w_{y}\right)\,dx\,dy.
\]
By using the Cauchy--Schwarz and Young inequalities we have $\mathcal{F}(v)\leq\mathcal{F}(w)$,
and the uniqueness follows from the strict convexity of $\mathcal{F}$. The rest of the formulas in the statement,
as well as \eqref{eq.36}, follow from \cite[Theorem~1.1]{Stinga-Torrea}.
\end{proof}
\begin{definition}
{For any $u\in \mathcal{H}_\e(\Omega)$ such that
$\displaystyle\int_\Omega u\,dx=0$, we will call the solution $v$ to problem \eqref{eq.4}
the $\e$--Neumann harmonic extension of $u$ and we write $v=E^\e(u)$.}
\end{definition}

\begin{remark}[Extensions of functions with nonzero average]\label{remarknonzeroav}
\rm{Observe that if $\displaystyle\int_\Omega u\,dx\neq 0$ then the function
\begin{equation}\label{extnonzeraver}
v(x,y)=\sum_{k=0}^{\infty}e^{-y(\e \lambda_k)^{1/2}}
\langle u,\varphi_k\rangle_{L^2(\Omega)}\varphi_{k}(x).
\end{equation}
is not in $L^2(\mathcal{C})$ in general but only its gradient $\nabla v\in L^2(\mathcal{C})$.
In order to give a suitable definition of the $\e$--Neumann
harmonic extension of $u$, we first solve the extension problem \eqref{eq.4}
with initial data $\tilde{u}=u-u_\Omega$, where $u_\Omega$ denotes the integral average of $u$ over $\Omega$,
see \eqref{eq:mean}, in order to find a function $\tilde{v}=E^{\e}(\tilde{u})$. {Then} we define
$$v{\equiv}E^{\e}(u):=E^{\e}(\tilde{u})+u_\Omega,$$
which clearly coincides with \eqref{extnonzeraver}.\\
Using the fact that the fractional {$\e$--}Neumann Laplacian
does not see constants, we have $$(-\e\Delta_N)^{1/2}u=(-\e\Delta_N)^{1/2}\tilde{u}=
{-\lim_{y\to0^+}\tilde{v}_y=
-\lim_{y\to0^+}v_y,\quad\hbox{in}~\mathcal{H}_\e(\Omega)'.}$$}
\end{remark}

Notice that the extension result {gives} in particular that
\begin{equation}
u\in\mathcal{H}_\e(\Omega),~\int_\Omega u\,dx=0, \quad
\hbox{implies}\quad u=\tr_{\Omega}v,~\hbox{for some}~v\in H^{1}(\mathcal{C}),\int_\Omega v(x,y)\,dx=0,~y\geq0.\label{eq.30}
\end{equation}
Now two questions come into our attention:
\begin{enumerate}[(i)]
\item Is the opposite implication in \eqref{eq.30} true?
\item Is there any integral characterization of the space $\mathcal{H}_\e(\Omega)$?
\end{enumerate}
The answer to question (i) is affirmative, and it will follow by an application of the trace inequality
of Lemma \ref{thm:traces}.
The answer to the second question is the identity $\mathcal{H}_\e(\Omega)=H^{1/2}(\Omega)$,
see Theorem \ref{thm:trace and H Omega}.

%%%%%%%%%%%%%%%%%%%%%%%%%%%%%%%%%%%%%%%%%%%%%%%%%%%%%%
\subsection{Characterization of $\mathcal{H}_\e(\Omega)=\Dom((-\e\Delta_N)^{1/2})$}
%%%%%%%%%%%%%%%%%%%%%%%%%%%%%%%%%%%%%%%%%%%%%%%%%%%%%%

Consider next the following fractional Sobolev space on the basis $\Omega\times\left\{0\right\}$ of $\mathcal{C}$:
\[
H^{1/2}(\Omega)=\left\{u\in L^{2}(\Omega):\|u\|^2_{H^{1/2}(\Omega)}:=
\|u\|_{L^2(\Omega)}^{2}+[u]^2_{H^{1/2}(\Omega)}<\infty\right\},
\]
where
\[
[u]^2_{H^{1/2}(\Omega)}:=\int_{\Omega}\int_{\Omega}\frac{|u(x)-u(y)|^{2}}{|x-y|^{n+1}}\,dx\,dy.
\]
We denote by $H^{-1/2}(\Omega)$ the dual space of the Hilbert space $H^{1/2}(\Omega)$.

Using again the semigroup language we can prove the following result.

\begin{theorem}[Domain of $(-\e\Delta_N)^{1/2}$]\label{thm:trace and H Omega}
For each $\e>0$ we have
$$\mathcal{H}_\e(\Omega)=H^{1/2}(\Omega),$$
as Hilbert spaces. In particular,
$$\|(-\e\Delta_N)^{1/4}u\|_{L^2(\Omega)}^2=\sum_{k=1}^\infty(\e\lambda_k)^{1/2}|\langle u,\varphi_k\rangle_{L^2(\Omega)}|^2
\sim \e^{1/2}[u]_{H^{1/2}(\Omega)}^2,$$
with equivalence constants depending only on $\Omega$ and $n$.
\end{theorem}

\begin{proof}
The proof is divided into three steps.

\smallskip

\noindent {\sc Step 1.} We claim that $u\in \mathcal{H}_\e(\Omega)$ if and only if $u\in L^2(\Omega)$ and the function
\[
I_{y}(u)=\frac{1}{y}\left[\langle u,u\rangle_{L^2(\Omega)}-\langle u,e^{-y(-\e\Delta_{N})^{1/2}}u
\rangle_{L^2(\Omega)}\right]
\]
is uniformly bounded in $y>0$; in such a case,
\[
\lim_{y\rightarrow0^+}I_{y}(u)=\sup_{y>0} I_{y}(u)=\langle(-\e\Delta_{N})^{1/2}u,u\rangle.
\]
To prove this, observe that
\[
I_{y}(u)=\sum_{k=0}^{\infty}\frac{1-e^{-y(\e\lambda_{k})^{1/2}}}{y}|\langle
u,\varphi_{k}\rangle_{L^2(\Omega)}|^{2}=\sum_{k=1}^{\infty}
\frac{1-e^{-y(\e\lambda_{k})^{1/2}}}{y(\e\lambda_{k})^{1/2}}(\e\lambda_{k})^{1/2}|
\langle u,\varphi_{k}\rangle_{L^2(\Omega)}|^{2}.
\]
The function $(1-e^{-y(\e\lambda_{k})^{1/2}})/(y(\e\lambda_{k})^{1/2})$
converges increasingly to 1 as $y\rightarrow0^+$. Then, if $u\in \mathcal{H}_\e(\Omega)$ we have
$I_{y}(u)\leq\|(-\e\Delta_N)^{1/4}u\|_{L^2(\Omega)}^2$,
thus $I_{y}(u)$ is uniformly bounded in $y>0$. Viceversa, if $I_{y}(u)$
is bounded uniformly with respect to $y$ it clearly follows that
$u\in\mathcal{H}_\e(\Omega)$. Apart from that, we have
\[
\sup_{y>0}I_{y}(u)=\lim_{y\rightarrow0^+}I_{y}(u)=\sum_{k=0}^{\infty}(\e\lambda_{k})^{1/2}|\langle
u,\varphi_{k}\rangle_{L^2(\Omega)}|^{2}=\langle(-\e\Delta_{N})^{1/2}u,u\rangle.
\]

\smallskip

\noindent {\sc Step 2.} For each $u\in L^2(\Omega)$,
\begin{equation}\label{eq.15}
yI_y(u)=\frac{1}{2}\int_{\Omega}\int_{\Omega}(u(x)-u(z))^2\mathcal{P}_{\e^{1/2}y}(x,z)\,dx\,dz,
\quad y>0,
\end{equation}
where $\mathcal{P}_y(x,z)$ is the Poisson kernel \eqref{Poisson kernel}. Indeed, by the definition of $I_{y}(u)$ we have
\begin{align}
yI_y(u)&=\int_{\Omega}u^2(x)\,dx-\int_{\Omega}u(x)\,e^{-y(-\e\Delta_{N})^{1/2}}u(x)\,dx\nonumber\\
&=\int_{\Omega}u^2(x)\,dx-\int_{\Omega}\int_{\Omega}u(x)\mathcal{P}_{\e^{1/2}y}(x,z)u(z)\,dx\,dz\nonumber\\
&=\int_{\Omega}\int_{\Omega}\left(u^2(x)-u(x)u(z)\right)\mathcal{P}_{\e^{1/2}y}(x,z)\,dx\,dz\label{eq.14},
\end{align}
where we have used that
\[
\int_{\Omega}\mathcal{P}_{y}(x,z)\,dz=1,\quad\hbox{for all}~x\in\Omega,~y>0,
\]
which follows by applying \eqref{integral heat kernel} to \eqref{Poisson kernel}.
Exchanging $x$ with $z$ in \eqref{eq.14} above, by Fubini's Theorem and the symmetry of the Poisson kernel
$\mathcal{P}_{y}(x,z)=\mathcal{P}_{y}(z,x)$, we get
\[
yI_y(u)=\int_{\Omega}\int_{\Omega}\left(u^2(z)-u(x)u(z)\right)\mathcal{P}_{\e^{1/2}y}(x,z)\,dx\,dz.
\]
Thus adding this equation to \eqref{eq.14} we arrive to \eqref{eq.15}.

\smallskip

\noindent {\sc Step 3. Conclusion.} From the subordination formula
\eqref{Poisson kernel} and the two sided estimates in \eqref{Saloff-Coste},
\begin{equation}
\mathcal{P}_{\e^{1/2}y}(x,z)\sim\frac{\e^{1/2}y}{(\e y^2+|x-z|^2)^\frac{n+1}{2}},
\quad\hbox{for all}~x,z\in\Omega,~{0<y<1},\label{eq.17}
\end{equation}
{while $|\mathcal{P}_{\e^{1/2}y}(x,z)|\leq M$, for all $x,z\in\Omega$, $y>1$.}
Then, by Step 2,
\[
I_{y}(u)\sim\int_{\Omega}\int_{\Omega}\frac{\e^{1/2}(u(x)-u(z))^2}{(\e y^2+|x-z|^2)^{\frac{n+1}{2}}}\,dx\,dz,
\]
for all $u\in L^2(\Omega)$. Finally, from Step 1, $u\in \mathcal{H}_\e(\Omega)$ if and only if
\begin{align*}
\sum_{k=1}^{\infty}(\e\lambda_{k})^{1/2}|\langle u,\varphi_{k}\rangle_{L^2(\Omega)}|^2
&= \langle(-\e\Delta_{N})^{1/2}u,u\rangle=\lim_{y\rightarrow0^+}I_{y}(u)\\
&\sim\e^{1/2}\int_{\Omega}\int_{\Omega}\frac{(u(x)-u(z))^2}{|x-z|^{n+1}}\,dx\,dz,
\end{align*}
namely, if and only if $u\in H^{1/2}(\Omega)$.
\end{proof}

%%%%%%%%%%%%%%%%%%%%%%%%%%%%%%%%%%%%%%%%%%%%%%%%%%%%%%
\subsection{Trace inequality}
%%%%%%%%%%%%%%%%%%%%%%%%%%%%%%%%%%%%%%%%%%%%%%%%%%%%%%

Let us define the space $\mathsf{H}^{\e}(\mathcal{C})$ as the completion of
$H^{1}(\mathcal{C})$ under the scalar product
\begin{equation}\label{eq:dos estrellas}
(v,w)_{\e}=\iint_{\mathcal{C}}\big(\e\nabla_{x} v\cdot\nabla_{x} w+
 v_{y}w_{y}\big)\,dx\,dy+\int_{\Omega}(\tr_{\Omega}v)(\tr_{\Omega}w)\,dx.
\end{equation}
We denote by $\|\cdot\|_\e$ the associated norm:
\begin{equation}\label{eq:trece estrellas}
\|v\|_{\e}^2=\iint_{\mathcal{C}}\big(\e|\nabla_{x} v|^2+
 v_{y}^2\big)\,dx\,dy+\int_{\Omega}(\tr_{\Omega}v)^2\,dx.
 \end{equation}
Notice that, for each $\e>0$,
\[
H^{1}(\mathcal{C})\subset\mathsf{H}^{\e}(\mathcal{C}),
\]
as Hilbert spaces, where the inclusion is strict, since constant functions belong to
$\mathsf{H}^{\e}(\mathcal{C})$ but not to $H^{1}(\mathcal{C})$. Let us point out that for a
finite-height cylinder
$$\mathcal{C}_{k}=\Omega\times(0,k),\quad\hbox{for}~k>0,$$
the following trace inequality
\begin{equation}
\|v\|_{L^{2}(\mathcal{C}_{k})}^{2}\leq C\left(\|\tr_{\Omega} v\|_{L^{2}(\Omega)}^{2}
+\|\nabla_{x,y} v\|_{L^{2}(\mathcal{C}_{k})}^{2}\right),\label{equivnormH}
\end{equation}
holds for all $v\in H^{1}(\mathcal{C}_{k})$, where the constant
$C$ depends only on $\Omega,k$ and $n$.
Indeed, this result can be proved by contradiction and applying
the Rellich--Kondrachov theorem. Then, by \eqref{equivnormH}, we easily infer that
$H^{1}(\mathcal{C}_{k})=\mathsf{H}^{\e}(\mathcal{C}_{k})$ as Hilbert spaces.
According to Remark \ref{remarknonzeroav}, we notice also that in general the $\e$--Neumann extension $v$ of a nonzero average function
$u\in\mathcal{H}_{\e}(\Omega)$ is in $\mathsf{H}^{\e}(\mathcal{C})$ but \emph{not} in $H^{1}(\mathcal{C})$.
Our aim is now to show that there is a trace operator
defined on the whole $\mathsf{H}^{\e}(\mathcal{C})$.

\begin{lemma}[Traces of functions in $\mathsf{H}^\e(\mathcal{C})$]\label{thm:traces}
For all $v\in H^{1}(\mathcal{C})$ one has
\begin{equation}\label{eq.20}
\begin{aligned}
\|(-\e\Delta_{N})^{1/4} v(x,0)\|_{L^{2}(\Omega)}^2&=\sum_{k=1}^\infty(\e\lambda_k)^{1/2}
|\langle \tr_\Omega v,\varphi_k\rangle_{L^2(\Omega)}|^2 \\
&\leq\iint_{\mathcal{C}}\big(\e|\nabla_{x}v|^{2}+|v_{y}|^{2}\big)\,dx\,dy.
\end{aligned}
\end{equation}
In particular, equality holds in \eqref{eq.20} if $v=E^{\e}(\tr_{\Omega}v)$.
Moreover, for each $\e>0$, there is a unique bounded linear operator
$T^{\e}:\mathsf{H}^{\e}(\mathcal{C})\rightarrow \mathcal{H}_\e(\Omega)$ such that $T^{\e}v(x)=v(x,0)$ if $v\in H^{1}(\mathcal{C})$
and, in particular,
\begin{equation}\label{traceemb}
\| T^{\e}v\|_{\mathcal{H}_\e(\Omega)}\leq \|v\|_\e.
\end{equation}
\end{lemma}

\begin{proof}
By Theorem \eqref{thm:trace and H Omega} we have that $v(\cdot,0)\in H^{1/2}(\Omega)=\mathcal{H}_\e(\Omega)$. Take $\tilde{u}=v(\cdot,0)-v(\cdot,0)_\Omega$, and consider the function
\[
\tilde{v}(x,y)=v(x,y)-v(\cdot,y)_\Omega
\]
Then $\tilde{v}\in H^{1}(\mathcal{C})$ and $\tilde{v}(x,0)=\tilde{u}(x)$, with $\tilde{u}_{\Omega}=0$. Thus from the fact that
 $\e$--Neumann extension {$E^\e(\tilde{u})$} of $\tilde{u}$
is an energy minimizer (see Theorem \ref{extensth}) and formula \eqref{eq.32}, we have
\begin{align*}
\iint_{\mathcal{C}}\big(\e|\nabla_{x}v|^{2}+|v_{y}|^{2}\big)\,dx\,dy&\geq\iint_{\mathcal{C}}\big(\e|\nabla_{x}\tilde{v}|^{2}+|\tilde{v}_{y}|^{2}\big)\,dx\,dy\\
&\geq \iint_{\mathcal{C}}\big(\e|\nabla_{x}E^\e(\tilde{u})|^{2}+|E^\e(\tilde{u})_{y}|^{2}\big)\,dx\,dy\\
&=\|(-\e\Delta_{{N}})^{1/4}\tilde{u}\|^{2}_{L^{2}(\Omega)}=\|(-\e\Delta_{{N}})^{1/4}u\|^{2}_{L^{2}(\Omega)}
\end{align*}
that is estimate \eqref{eq.20}.
\normalcolor
Now using \eqref{traceemb}, it is always possible to extend the trace operator on $\Omega\times\left\{0\right\}$
from $H^{1}(\mathcal{C})$ into $\mathsf{H}^{\e}(\mathcal{C})$. Indeed, if $v\in
\mathsf{H}^{\e}(\mathcal{C})$ and $\left\{v_{k}\right\}_{k\in\N}\subset H^{1}(\mathcal{C})$
is a sequence converging to $v$ in $\mathsf{H}^{\e}(\mathcal{C})$, we define the operator
\[
T^{\e}v:=\lim_{k\rightarrow\infty} v_{k}(\cdot,0),
\]
where the limit is taken in $\mathcal{H}_\e(\Omega)$.
Notice that this limit exists because, by inequality \eqref{traceemb},
\[
\|v_{k}(\cdot,0)-v_{l}(\cdot,0)\|_{\mathcal{H}_\e(\Omega)}\leq
\|v_{k}-v_{l}\|_\e\rightarrow0,\quad k,l\to\infty,
\]
and obviously the definition of $T^\e$ over $\mathsf{H}^{\e}(\mathcal{C})$ does not depend on the chosen sequence.
Since
\[
\|v_{k}(\cdot,0)\|_{\mathcal{H}_\e(\Omega)}\leq\|v_{k}\|_\e,\quad\hbox{for all}~k,
\]
it is clear that $T^{\e}$ is linear and bounded. Finally this operator is unique,
due to the density of $H^{1}(\mathcal{C})$ in $\mathsf{H}^{\e}(\mathcal{C})$.
\end{proof}

%%%%%%%%%%%%%%%%%%%%%%%%%%%%%%%%%%%%%%%%%%%%%%%%%%%%%%
\subsection{Compactness of the trace operator $T^\e$}
%%%%%%%%%%%%%%%%%%%%%%%%%%%%%%%%%%%%%%%%%%%%%%%%%%%%%%

The following remark follows from Theorem \ref{thm:trace and H Omega}
and Lemma \ref{thm:traces}.

\begin{remark}\label{equivHeps}
{\rm For all $u\in \mathcal{H}_{\varepsilon}(\Omega)$, we have
\[
\|u\|_{\mathcal{H}_{\varepsilon}(\Omega)}\geq \mathsf{C}(\e)\|u\|_{H^{1/2}(\Omega)},
\]
where
\[
\mathsf{C}(\e)=c\sqrt{\min\left\{1,\e^{1/2}\right\}}
\]
being $c$ a positive constant depending only on $n$ and $\Omega$. In particular,
$$T^\e:\mathsf{H}^\e(\mathcal{C})\to H^{1/2}(\Omega),$$
and for every $v\in\mathsf{H}^\e(\mathcal{C})$,
\begin{equation}\label{traceembH}
\mathsf{C}(\e)\|T^\e v\|_{H^{1/2}(\Omega)}\leq \|v\|_\e.
\end{equation}}
\end{remark}
Of course from inequality \eqref{traceembH} and the Sobolev embedding (see \cite{Adams})
$$H^{1/2}(\Omega)\hookrightarrow L^{\frac{2n}{n-1}}(\Omega)$$ it follows that, for all $v\in
\mathsf{H}^{\e}(\mathcal{C})$,
\begin{equation*}
\|T^{\e}v\|_{L^{\frac{2n}{n-1}}(\Omega)}\leq C_0\|T^{\e}v\|_{H^{1/2}(\Omega)}
\leq C_0\mathsf{C}(\e)^{-1}\|v\|_\e.
\end{equation*}
In conclusion,
\begin{equation}\label{traceeqcomp}
C_0^{-1}\mathsf{C}(\e)\|T^{\e}v\|_{L^{\frac{2n}{n-1}}(\Omega)}\leq\|v\|_\e,
\end{equation}
where the constant $C_0$ does not depend on $\e$. Inequality \eqref{traceeqcomp}
will be called the \emph{trace embedding inequality} for the
space $\mathsf{H}^{\e}(\mathcal{C})$.
An immediate consequence is the following
compactness result for traces of functions in $\mathsf{H}^{\e}(\mathcal{C})$,
which follows from the compact embedding $
H^{1/2}(\Omega)\subset\subset L^q(\Omega)$ for $1{\leq}q<2n/(n-1)$.

\begin{corollary}\label{compactness of the trace}
Fix $\e>0$. Then
$$T^{\e}(\mathsf{H^{\e}}(\mathcal{C}))\subset\subset L^{q}(\Omega),\quad
\hbox{for all}~1\leq q<2n/(n-1).$$
\end{corollary}

%%%%%%%%%%%%%%%%%%%%%%%%%%%%%%%%%%%%%%%%%%%%%%%%%%%%%%
\section{The linear Neumann problem: regularity}\label{Section:lineal}
%%%%%%%%%%%%%%%%%%%%%%%%%%%%%%%%%%%%%%%%%%%%%%%%%%%%%%

In this section we study regularity properties
of solutions $u$ to the linear fractional Neumann problem \eqref{eq.3},
where $f\in H^{-1/2}(\Omega)$. We understand \eqref{eq.3} in the sense that
\begin{equation}\label{numero}
(-\e\Delta_N)^{1/2}u+u=f,\quad\hbox{in}~H^{-1/2}(\Omega).
\end{equation}
The linear problem \eqref{eq.3} has a unique explicit solution $u$ in $H^{1/2}(\Omega)=\mathcal{H}_\e(\Omega)$ which is given
in terms of the Neumann eigenfunctions of the Laplacian \eqref{eq.1}. Indeed, by virtue of Theorem \ref{thm:trace and H Omega}, we see
that $H^{-1/2}(\Omega)$ coincides with the dual space of $\mathcal{H}_\e(\Omega)$, for each $\e>0$.
Then the Riesz representation
theorem implies that $f\in\mathcal{H}_\e(\Omega)'$ can be written
in a unique way as $f=\sum_{k=0}^\infty f_k\varphi_k$ in $H^{-1/2}(\Omega)$,
where $\sum_{k=0}^\infty\big((\e\lambda_k)^{1/2}+1\big)^{-1}f_k^2<\infty$.
Then the solution $u$ {is given by}
\begin{equation}\label{sollinprob}
u(x)=\sum_{k=0}^{\infty}\frac{f_k}{(\e\lambda_k)^{1/2}+1}\varphi_{k}(x)\in H^{1/2}(\Omega).
\end{equation}
It is readily verified that $u$ solves \eqref{eq.3} and it is unique because
of the orthogonality of the eigenfunctions $\varphi_k$.

\begin{remark}[Regularity in $H^{1}(\Omega)$]
{\rm If $f\in L^{2}(\Omega)$, the solution $u$ to \eqref{eq.3} is actually in $H^{1}(\Omega)$.  Indeed, according to \eqref{sollinprob} we have
\[
u=\sum_{k=0}^{\infty}\frac{\langle f,\varphi_k\rangle_{L^{2}(\Omega)}}{(\e\lambda_k)^{1/2}+1}\varphi_{k},
\]
and
\[
\sum_{k=1}^{\infty}\lambda_{k}\frac{|\langle f,\varphi_k\rangle_{L^{2}(\Omega)}|^2}{|(\e\lambda_k)^{1/2}+1|^{2}}
\leq\frac{1}{\e}\sum_{k=1}^\infty|\langle f,\varphi_k\rangle_{L^{2}(\Omega)}|^2<\infty.
\]
Hence, by the spectral representation of $H^{1}(\Omega)$ provided by \eqref{eq.H1}, $u\in H^{1}(\Omega)$.}
\end{remark}

By virtue of the extension problem for $(-\e\Delta_N)^{1/2}$, {we} can show that $u$ is the trace over $\Omega$
of the solution $v$ to the following extension problem:
\begin{equation}\label{eq.6}
\begin{cases}
  \e\Delta_{x}v+v_{yy}=0,&\hbox{in}~\mathcal{C},\\
 \partial_{\nu}v=0,&\hbox{on}~\partial_L\mathcal{C},\\
 -v_y(x,0)+v(x,0)=f,&\hbox{in}~H^{-1/2}(\Omega).
\end{cases}
\end{equation}
Indeed, by multiplying formally the equation above by a test function and integrating by parts, we can give the following suitable definition of weak solution.

\begin{definition}\label{weaksollinprob}Let $f\in H^{-1/2}(\Omega)$. We say that a function $v\in
\mathsf{H}^\e(\mathcal{C})$ is a weak solution to \eqref{eq.6} if
\begin{equation}\iint_{\mathcal{C}}\big(\e\nabla_{x} v\cdot\nabla_{x} w+
 v_{y}w_{y}\big)\,dx\,dy+\int_{\Omega}(\tr_{\Omega}v)(\tr_{\Omega}w)\,dx=\langle f,\tr_\Omega w\rangle,\label{eq.33}\end{equation}
 for every $w\in \mathsf{H}^\e(\mathcal{C})$.
 \end{definition}

 Notice that $T^\e w\in H^{1/2}(\Omega)$ for all $w\in \mathsf{H}^\e(\mathcal{C})$ so the right hand side in \eqref{eq.33}
is well defined.

The following Lemma provides the explicit form of the weak solution to problem \eqref{eq.6}.

\begin{lemma}\label{existweaksollin}
The function
$$v(x,y)=\sum_{k=0}^\infty e^{-y(\e\lambda_k)^{1/2}}\frac{f_k}{(\e\lambda_k)^{1/2}+1}\varphi_{k}(x),$$
is the unique weak solution $v\in \mathsf{H}^\e(\mathcal{C})$ of \eqref{eq.6}. Moreover,
$T^\e v$
coincides with the unique solution $u\in H^{1/2}(\Omega)$ to the linear problem \eqref{eq.3}.
\end{lemma}

\begin{proof}
Simple computations as in the proof of Theorem \ref{extensth} show that $v\in
\mathsf{H}^\e(\mathcal{C})$.
Let $w\in H^1(\mathcal{C})$ be a test function. We can write, for almost every $y\geq0$,
$$w(x,y)=\sum_{k=0}^\infty \langle w(\cdot,y),\varphi_k\rangle_{L^2(\Omega)}\varphi_k=:\sum_{k=0}^\infty
w_k(y)\varphi_k(x).$$
Then by using \eqref{eq.2.2}, the definition of $v$ and \eqref{eq:estrella}, we have for almost every $y>0$
$$\int_\Omega v_{yy}w\,dx=\sum_{k=0}^\infty \e\lambda_ke^{-y(\e\lambda_k)^{1/2}}\frac{f_k}{(\e
 \lambda_k)^{1/2}+1}w_{k}(y)=\e\int_\Omega\nabla_xv\cdot\nabla_xw\,dx.$$
Integrating this identity in $y$ and applying integration by parts,
\begin{align*}
 \iint_{\mathcal{C}}\big(\e\nabla_xv\cdot\nabla_xw+v_yw_y\big)\,dx\,dy &=
-\lim_{y\to0^+}\int_\Omega v_y(x,y)w(x,y)\,dx \\
 &=\sum_{k=0}^\infty(\e\lambda_k)^{1/2}\frac{f_k}
 {(\e\lambda_k)^{1/2}+1}w_k(0) \\
 &= \sum_{k=0}^\infty f_kw_k(0)
 -\sum_{k=0}^\infty\frac{f_k}{(\e\lambda_k)^{1/2}+1}w_k(0) \\
 &= \langle f,\tr_\Omega w\rangle-\int_\Omega(\tr_\Omega v)(\tr_\Omega w)\,dx.
 \end{align*}
 Hence a density argument shows that
 the function $v$ is a weak solution in the sense of Definition \ref{weaksollinprob}.
Uniqueness can be proved in two ways as explained in the proof of Theorem \ref{extensth}.
\end{proof}

%%%%%%%%%%%%%%%%%%%%%%%%%%%%%%%%%%%%%%%%%%%%%%%%%%%%%%
\subsection{Harnack estimates}
%%%%%%%%%%%%%%%%%%%%%%%%%%%%%%%%%%%%%%%%%%%%%%%%%%%%%%

When studying the spike layer solutions to the semilinear problem we will need the following
Harnack estimate.

\begin{theorem}[Harnack inequality]\label{Harnack}
Let $u\in H^{1/2}(\Omega)$ be a nonnegative solution to
\begin{equation*}
\begin{cases}
 (-\e\Delta)^{1/2}u+c(x)u=0,&\hbox{in}~\Omega,\\
 \partial_\nu u=0,&\hbox{on}~\partial\Omega,
\end{cases}
\end{equation*}
where $c(x)$ is a bounded function. Then for any $R>0$ there exists a positive constant
$C$ depending only on $n$, $\Omega$ and $R(\|c\|_{L^\infty(\Omega)}/\e)^{1/2}$
such that for any ball $B=B(x_0,R)$, with $x_0\in\overline{\Omega}$,
$$\sup_{B\cap\Omega}u\leq C\inf_{B\cap\Omega}u.$$
\end{theorem}

\begin{proof}
The proof relies on the extension problem and the reflection method.
Let $v$ be a nonnegative solution to the extension problem
\begin{equation*}
\begin{cases}
 \e\Delta_xv+v_{yy}=0,&\hbox{in}~\mathcal{C},\\
 \partial_\nu v=0,&\hbox{on}~\partial_L\mathcal{C},\\
 v_y=c(x)v(x,0),&\hbox{on}~\Omega.
\end{cases}
\end{equation*}
If the ball $B$ lies in the interior of the domain $\Omega$, then
we can repeat the proof of \cite[Lemma~2.5]{Cabre-Sola}. If $x_0$ lies in the boundary of
$\Omega$, we first extend a suitable modification of
$v$ to the whole cilinder $\Omega\times\mathbb{R}$ as in \cite[Lemma~2.5]{Cabre-Sola}.
Then we can follow the idea of \cite[Proof~of~Lemma~4.3]{Lin-Ni-Takagi}.
We flatten the boundary around $x_0$ and reflect the solution
to have a solution to a linear elliptic equation on a cilinder $B_r\times\R$.
The interior Harnack inequality for linear equations applies and we get the
result on the boundary.
\end{proof}

%%%%%%%%%%%%%%%%%%%%%%%%%%%%%%%%%%%%%%%%%%%%%%%%%%%%%%
\subsection{Regularity}
%%%%%%%%%%%%%%%%%%%%%%%%%%%%%%%%%%%%%%%%%%%%%%%%%%%%%%

Here we establish the basic regularity properties for the linear problem \eqref{eq.3}.
Again the use of the Neumann heat semigroup plays a central role.

\begin{theorem}[Regularity estimates]\label{Thm:regularity}
Let $u\in H^{1/2}(\Omega)$ and $f\in H^{-1/2}(\Omega)$ be such that \eqref{eq.3} holds in
the sense of \eqref{numero}.
Then the following assertions hold.
\begin{enumerate}[$(1)$]
 \item Let $1\leq p\leq\infty$ and $f\in L^p(\Omega)$. Then $u\in L^p(\Omega)$ and
 \begin{equation}\|u\|_{L^p(\Omega)}\leq C_{\e,n,p,\Omega}\,\|f\|_{L^p(\Omega)}\label{reg1}.\end{equation}
 Moreover:\\
 $(a_{1})$ if $1<p<n$, then $u\in L^q(\Omega)$ for $p\leq q\leq\frac{np}{n-p}$, and
 \begin{equation}\|u\|_{L^q(\Omega)}\leq C_{\e,n,p,q,\Omega}\,\|f\|_{L^p(\Omega)};\label{reg2}\end{equation}
  $(b_{1})$ if $n<p\leq\infty$, then $u\in L^\infty(\Omega)$ and
 \begin{equation}\|u\|_{L^\infty(\Omega)}\leq C_{\e,n,p,\Omega}\,\|f\|_{L^p(\Omega)}.\label{reg3}\end{equation}
 \item If $f\in L^p(\Omega)$ for $n\leq p<\infty$, then:\\
 $(a_{2})$ for $p>n$ we have $u\in C^{0,\alpha}(\overline{\Omega})$,
 where $\alpha:=1-\frac{n}{p}\in(0,1)$ and
 \begin{equation}\|u\|_{C^{0,\alpha}(\overline{\Omega})}\leq C_{\e,n,p,\Omega}\,\|f\|_{L^p(\Omega)};\label{reg4}\end{equation}
 $(b_{2})$ for $p=n$, we have $u\in BMO(\Omega)$ and
 \begin{equation}\|u\|_{BMO(\Omega)}\leq C_{\e,n,p,\Omega}\,\|f\|_{L^p(\Omega)}.\label{reg5}\end{equation}
$(c_{2})$ Besides, if $p\leq r<\infty$ we have $u\in L^{r}(\Omega)$ and the $L^{r}$- norm of $u$ can be added at the left-hand side of both inequalities  \eqref{reg4} and \eqref{reg5}, provided the constant at the right-hand side depends on $r$ too.
 \item If $f\in L^{\infty}(\Omega)$ then $u\in C^{0,\alpha}(\overline{\Omega})$ for any $0<\alpha<1$,
 and
 $$\|u\|_{C^{0,\alpha}(\overline{\Omega})}\leq C_{\e,n,\alpha,\Omega}\,\|f\|_{L^\infty(\Omega)}.$$
 \item If $f\in C^{0,\alpha}(\overline{\Omega})$ for $0<\alpha<1$, then $u\in C^{1,\alpha}(\overline{\Omega})$
 and
 $$\|u\|_{C^{1,\alpha}(\overline{\Omega})}\leq C_{\e,n,\alpha,\Omega}\,\|f\|_{C^{0,\alpha}(\overline{\Omega})}.$$
\end{enumerate}
\end{theorem}

\begin{remark}
{\rm If the constants $C$ in the estimates above are tracked down {along the proof},
we would readily see that they blow up when $\e\to0^+$.}
\end{remark}

\begin{proof}[Proof of Theorem \ref{Thm:regularity}]
Recall \eqref{sollinprob}.
If $f$ is also locally integrable then we can write
\begin{equation}\label{u(x)}
\begin{aligned}
 u(x) &= ((-\e\Delta_N)^{1/2}+I)^{-1}f(x)=\int_0^\infty e^{-t}e^{-t(-\e\Delta_N)^{1/2}}f(x)\,dt \\
 &= \int_0^\infty e^{-t}e^{-(t\e^{1/2})(-\Delta_N)^{1/2}}f(x)\,dt= \int_{\Omega}L(x,z)f(z)\,dz,
\end{aligned}
\end{equation}
where
$$L(x,z)=\int_0^\infty e^{-t}\mathcal{P}_{(\e^{1/2}t)}(x,z)\,dt,\quad x,z\in\Omega,$$
and $\mathcal{P}_t$ is the kernel of the Poisson semigroup for the Neumann Laplacian (see \eqref{Poisson kernel}). Now,
from \eqref{eq.17} and by using the change of variables $s=\frac{\e^{1/2}t}{|x-z|}$ it follows
that, for some finite constant $C_n$,
$$0\leq L(x,{z})\leq \frac{C_{{n,\e,\Omega}}}{|x-z|^{n-1}},\quad\hbox{for}~x,z\in\Omega.$$
The estimate above implies that
$$|u(x)|\leq {C_{n,\e,\Omega}}N\ast|f\chi_\Omega|(x),$$
where $N(x)=|x|^{1-n}$ is the kernel of the classical fractional integral of order $1$.

\

\noindent$(1)$ This result follows from Young's inequality for convolutions,
since $N(x)\chi_\Omega(x)\in L^1(\R^n)$. Moreover, inequality \eqref{reg2}
follows from the well known Hardy--Littlewood--Sobolev inequality for fractional integration, see \cite{Harboure} or \cite[Chapter~V]{Stein-Singular}.
Concerning \eqref{reg3}, notice that, from \eqref{eq.5} and \eqref{eq.17}, by H\"older's
inequality with $\frac{1}{p}+\frac{1}{q}=1$ and by applying the change of variables
$w=z/(\e^{1/2}t)$, we get {for $t<1$},
\begin{align*}
|e^{-(t\e^{1/2})(-\Delta_N)^{1/2}}f(x)|
&\leq C_{n,\Omega}\|f\|_{L^p(\Omega)}\left(\int_{\R^n}\frac{(\e^{1/2}t)^q}{((\e^{1/2}t)^2+|z|^2)^{q\frac{n+1}{2}}}\,dz
\right)^{1/q} \\
&=C_{\e,n,\Omega}\,t^{\frac{n-nq}{q}}\|f\|_{L^p(\Omega)}
\left(\int_{\R^n}\frac{1}{(1+|w|^2)^{q\frac{n+1}{2}}}\,dw
\right)^{1/q}\\
&=C_{\e,n,p,\Omega}\,t^{-n/p}\|f\|_{L^p(\Omega)}.
\end{align*}
{On the other hand, for $t\geq1$, H\"older's inequality readily
 gives $|e^{-(t\e^{1/2})(-\Delta_N)^{1/2}}f(x)|\leq
C\|f\|_{L^p(\Omega)}$}.
We can now estimate in the second to last identity of \eqref{u(x)}:
$$|u(x)|\leq C_{\e,n,p,\Omega}\,\|f\|_{L^p(\Omega)}\int_0^\infty e^{-t}{\min\{t,1\}}^{-n/p}\,dt=C_{\e,n,p,\Omega}\,\|f\|_{L^p(\Omega)},$$
which gives the conclusion (notice that $-n/p+1>0$).

\

\noindent $(2)$ The classical result by S. Campanato in \cite{Campanato}
establishes that $C^\alpha(\overline{\Omega})$ for $\alpha\in(0,1)$ coincides, with equivalent seminorms, with the space
$BMO^\alpha(\Omega)$. The fact that  $u\in BMO^\alpha(\Omega)$, $0\leq \alpha<1$, can be established as in the classical case
of the fractional integral on $\R^n$, see for example \cite{Harboure}. For completeness we provide the proof,
where the Neumann heat kernel plays a useful role again.
Let $f$ be as in the hypothesis. By $(1)$
it follows that $u\in L^p(\Omega)$, hence $u$ is locally integrable.
Let $x_0\in\Omega$ and let $B$ be a ball centered at $x_0$ of radius $r_B$. We decompose $f$ as $f=f_1+f_2$, where
$f_1=f\chi_{(2B)\cap\Omega}$ and $2B$ is the ball centered at $x_0$ with radius $2r_B$.
By \eqref{u(x)} we get $u(x)=u_1(x)+u_2(x)$, where $u_i$ corresponds to the integral of $f_i(y)$
against $L(x,y)$, $i=1,2$.
We denote by $(u)_D$ the integral mean of $u$ over a set $D$. We can write
\begin{align}
\frac{1}{|B\cap\Omega|}&\int_{B\cap\Omega}|u(x)-(u)_{B\cap\Omega}|\,dx\nonumber\\
&\leq \frac{2}{|B\cap\Omega|}\int_{B\cap\Omega}|u_1(x)|\,dx+\frac{1}{|B\cap\Omega|}
\int_{B\cap\Omega}|u_2(x)-(u_2)_{B\cap\Omega}|\,dx=:I+II.\label{BMO1}
\end{align}
Let us choose $1<\gamma<n$ and $q=\frac{n\gamma}{n-\gamma}$. By assertion $(a_1)$
and H\"older's inequality with $r=p/\gamma$ and $r'=\frac{p}{p-\gamma}$,
\begin{align}
I &\leq \left(\frac{2}{|B\cap\Omega|}\int_{B\cap\Omega}|u_1(x)|^q\right)^{1/q}
\leq \frac{C}{|B\cap\Omega|^{1/q}}\|f_1\|_{L^\gamma(\Omega)}\nonumber\\
&\leq \frac{C}{|B\cap\Omega|^{\frac{n-\gamma}{n\gamma}}}
\|f\|_{L^p(B\cap\Omega)}|B\cap\Omega|^{\frac{p-\gamma}{p\gamma}}
= C\|f\|_{L^p(\Omega)}|B\cap\Omega|^{\frac{1}{n}-\frac{1}{p}}\nonumber\\
&= C\|f\|_{L^p(\Omega)}|B\cap\Omega|^{\alpha/n},\label{BMO2}
\end{align}
with $C$ as in $(a_1)$ and $\alpha=1-n/p$. On the other hand,
\begin{equation}II\leq \frac{1}{|B\cap\Omega|^2}
\int_{B\cap\Omega}\int_{B\cap\Omega}\int_{(2B)^c\cap\Omega}|f_2(z)||L(x,z)-L(y,z)|\,dz\,dy\,dx.\label{BMO3}\end{equation}
Since $x,y\in B$ and $z\in(2B)^c$, by the mean value theorem applied in the definition of $L$
(use \eqref{Wang-Yan} into the second identity of \eqref{Poisson kernel}) it follows that
$$|L(x,z)-L(y,z)|\leq C\frac{\e^{-1/2}r_B}{|x_0-z|^{n}}.$$
Thus H\"older's inequality yields
\begin{align*}
II &\leq Cr_B\int_{(2B)^c\cap\Omega}\frac{|f_2(z)|}{|x_0-z|^{n}}\,dz\leq Cr_B\|f\|_{L^p(\Omega)}
\left(\int_{(2B)^c}\frac{1}{|x_0-z|^{\frac{np}{p-1}}}\,dz\right)^{\frac{p-1}{p}}=C\|f\|_{L^p(\Omega)}r_B^{1-n/p}.
\end{align*}
Collecting \eqref{BMO1}, \eqref{BMO2} and \eqref{BMO3},
$$\frac{1}{|B\cap\Omega|}\int_{B\cap\Omega}|u(x)-(u)_{B\cap\Omega}|\,dx\leq C\|f\|_{L^p(\Omega)}|B\cap\Omega|^{\alpha/n},$$
so that $u\in BMO^\alpha(\Omega)$. Then using \eqref{reg3}, the estimates \eqref{reg4} and \eqref{reg5} follow.

Now take any $r\in[p,\infty)$ and let us prove that $u\in L^r(\Omega)$. If $p>n$ then
from \eqref{reg3} we clearly have $\|u\|_{L^{r}(\Omega)}\leq C_{\e,n,p,r,\Omega}\,\|f\|_{L^{p}(\Omega)}$. Assume now that $f\in L^n(\Omega)$. Therefore by \eqref{reg1} we get $u\in L^n(\Omega)$.
On the other hand, suppose that $r>n$. If we pick $p$ such that $r=\frac{np}{n-p}$ we have $p\in(n/2,n)$.
Then $f\in L^p(\Omega)$, so \eqref{reg2} gives $u\in L^{\frac{np}{n-p}}(\Omega)
=L^r(\Omega)$.

\

\noindent$(3)$ For any $\alpha\in(0,1)$, take $p>n$ such that $\alpha=1-n/p$ and use \eqref{reg4}.

\

\noindent$(4)$ If $f\in C^{0,\alpha}(\overline{\Omega})$, by property {(3)} we have that $u\in C^{0,\alpha}(\overline{\Omega})$.
 If $v=E^{\e}u$, then the rescaled function
\[
\widetilde{v}(x,y)=v(x,\e^{-1/2}y)
\]
solves the problem
$$\begin{cases}
\Delta_{x,y}\widetilde{v}=0,&\hbox{in}~\mathcal{C},\\
\partial_{\nu}\widetilde{v}=0,&\hbox{on}~\partial_{L}\mathcal{C},\\
-\widetilde{v}_y(x,0)=h(x),&\hbox{on}~\Omega,
\end{cases}$$
where
\[
h=\e^{-1/2}(f-u)\in C^{0,\alpha}(\overline{\Omega}).
\]
Consider the function
\[
w(x,y)=\int_{0}^{y}\widetilde{v}(x,t)\,dt,\quad y\geq0.
\]
Then we have $w(x,0)=0$, and
\[
\Delta_{x,y}w=\widetilde{v}_y+\int_{0}^{y}\Delta_{x}\widetilde{v}(x,t)\,dt.
\]
Thus by using the equation for $\widetilde{v}$ we have $(\Delta_{x,y} w)_{y}=0$ in $\mathcal{C}$.
This says that $\Delta_{x,y} w$ is constant as a function of $y$, so we can compute it
by taking its value at $\left\{y=0\right\}$. Observe that
$$\Delta_{x,y}w\big|_{y=0}=\widetilde{v}_y\big|_{y=0}=-h(x).$$
Hence $w$ solves
\begin{equation}\label{w equation}
\begin{cases}
-\Delta_{x,y}w  =h(x),&\hbox{in}~\mathcal{C},\\
\partial_{\nu} w=0,&\hbox{on}~\partial_{L}\mathcal{C},\\
w(x,0)=0&\hbox{on}~\overline{\Omega}.
\end{cases}
\end{equation}
Then we aim to study the boundary regularity of $w$. We take a point $(x_0,0)$ with $x_0\in\partial\Omega$. Without loss of generality we can assume that $x_0$ is the origin and that the exterior
normal to $\partial\Omega$ at $x_0$ is the vector $-e_n\in\R^n$.
 Since $\partial\Omega$ is smooth, we can describe it near $x_0$ as the graph of an $(n-1)$ variables differentiable
 map and then flatten it with a transformation $\Psi$. Let us call $z$ the new variables in the flat geometry, and let
 $\bar{w}(z,y)=w(x,y)$. Notice that when we flatten the boundary we are leaving the $y$ variable fixed. As in \cite[pp.~18--19]{Lin-Ni-Takagi}
 it can be verified that $\bar{w}(z,y)$ satisfies the following extension problem:
 $$\begin{cases}
  -L_z\bar{w}-\bar{w}_{yy}=\bar{h}(z),&\hbox{in}~(B_{\delta}\cap\{z_n>0\})\times(0,\infty),\\
  \partial_{\nu}\bar{w}=0,&\hbox{on}~(B_{\delta}\cap\{z_n=0\})\times[0,\infty),\\
  \bar{w}(z,0)=0,&\hbox{in}~B_{\delta}\cap\{z_n\geq0\},
 \end{cases}$$
for some small $\delta>0$, where $B_{\delta}=B_{\delta}(0)$ is the ball centered at $0$ of radius $\delta$. Here $L_z$ is a nondivergence form elliptic operator with
smooth coefficients. At this point, we could extend $\bar{w}$, the operator $L_z$ and the known term $\bar{h}$ to $B_{\delta}\times[0,\infty)$ by an even reflection to get another equation in nondivergence form posed in the whole ball $B_{\delta}$, and satisfied by the even reflection $\bar{w}_{ev}$ of $\bar{w}$. Then \cite[{Lemma~6.18}]{Gilbarg-Trudinger} gives that $\bar{w}_{ev}\in C^{2,\alpha}(B_{\delta}\times[0,\infty))$. This implies that $v\in C^{1,\alpha}({\overline{\mathcal{C}}})$.
\end{proof}

%%%%%%%%%%%%%%%%%%%%%%%%%%%%%%%%%%%%%%%%%%%%%%%%%%%%%%
\section{Existence of {nonconstant} least energy positive regular solutions to the fractional Neumann semilinear problem}\label{Section4}
%%%%%%%%%%%%%%%%%%%%%%%%%%%%%%%%%%%%%%%%%%%%%%%%%%%%%%

In this section we find {nonconstant}
 positive solutions $u\in H^{1/2}(\Omega)$ to the semilinear Neumann problem
\eqref{eq.8}--\eqref{nonlinearity}
and study regularity properties for small $\e$. Here we assume \eqref{eq.9},
that is, $p$ is strictly smaller than the \emph{critical Sobolev trace exponent} $(n+1)/(n-1)$.
As in the linear case, in general we understand \eqref{eq.8} as
$$(-\e\Delta_N)^{1/2}u+u=g(u),\quad\hbox{in}~H^{-1/2}(\Omega),$$
{where $g(t)=(t_{+})^{p}$, for all $t\in \R$}.
In order to define a solution, consider the semilinear extension problem
\begin{equation}\label{eq.10}
\begin{cases}
\e\Delta_xv+v_{yy}=0,&\hbox{in}~\mathcal{C},\\
\partial_{\nu} v=0,&\hbox{on}~\partial_{L}\mathcal{C},\\
-v_y(x,0)+v(x,0)=g(v(x,0)),&\hbox{in}~H^{-1/2}(\Omega).
\end{cases}
\end{equation}
For such a problem we have the following suitable definition of weak solution.

\begin{definition}\label{def:weak}
A function $v\in\mathsf{H}^\e(\mathcal{C})$ is a weak solution to \eqref{eq.10} if
for every $w\in\mathsf{H}^\e(\mathcal{C})$ we have
\begin{equation}
(v,w)_\e=\langle f^{\e}_v,T^\e w\rangle,\label{weakH}
\end{equation}
where $(\cdot,\cdot)_\e$ is the inner product in $\mathsf{H}^\e(\mathcal{C})$, see \eqref{eq:dos estrellas},
and $f_{v}^{\e}$ is the functional in $H^{-1/2}(\Omega)$ defined by
$$\langle f^{\e}_v,\varphi\rangle=\int_{\Omega}g(T^{\e} v)\varphi\,dx,\quad\hbox{for each}~\varphi
\in H^{1/2}(\Omega).$$
\end{definition}

Observe that $f_{v}^\e$ is indeed in $H^{-1/2}(\Omega)$ because, by H\"older's inequality and the boundedness of the trace operator $T^\e$ from $\mathsf{H}^\e(\mathcal{C})$ into $L^{2n/(n-1)}(\Omega)$,
$$\int_\Omega|g(T^{\e} v)|^{\frac{2n}{n+1}}\,dx\leq\int_\Omega|T^{\e} v|^{\frac{2n}{n+1}p}\,dx
\leq C\|T^{\e} v\|^{2np/(n+1)}_{L^{\frac{2n}{n-1}}(\Omega)},$$
for some constant $C$.

As in the linear case, according to what we explained in the Introduction, it is natural to give the following definition of weak solution to \eqref{eq.8}.

\begin{definition}
A function $u\in H^{1/2}(\Omega)$ is a weak solution to \eqref{eq.8} if $u=T^{\e}v$, where $v$
 solves \eqref{eq.10} {in the sense of Definition \ref{def:weak}}.
\end{definition}

We look for a {nonconstant}
weak solution $v$ to \eqref{eq.10} as a {nonconstant} critical
point over $\mathsf{H}^{\e}(\mathcal{C})$ of the functional $\mathcal{I}_{\e}$ (see \eqref{4 star}):
\begin{equation}\label{eq.11}
\mathcal{I}_{\e}(v)=\frac{1}{2}\|v\|^{2}_{\e}
-\int_{\Omega}G(T^{\e} v)\,dx,
\end{equation}
where $\|\cdot\|_\e$ is the norm in $\mathsf{H}^{\e}(\mathcal{C})$, see \eqref{eq:trece estrellas}, and
\[
G(t)=\int_{0}^{t}g(\xi)\,d\xi=\begin{cases}
\frac{1}{p+1}t^{p+1},&\hbox{if}~t\geq0,\\
0,&\hbox{if}~t\leq0.
\end{cases}
\]
The second term in the right hand side of \eqref{eq.11} is well defined because
of the fractional Sobolev embedding (notice that $2<p+1<2n/(n-1)$).
Then we have the following results, which can be seen as the fractional
version of \cite[Theorem 2]{Lin-Ni-Takagi}.

\begin{theorem}\label{TeoremaMountain}
There is at least one positive {nonconstant} solution $v_\e{\in C^{2,\alpha}(\mathcal{C})
\cap C^{1,\alpha}(\overline{\mathcal{C}})}$, {for $0<\alpha<1$}, to problem \eqref{eq.10},
{provided $\e>0$ is sufficiently small. In this case}
there exists a positive constant $C$, depending only on $p$ and $\Omega$, such that
 $$\mathcal{I}_{\e}(v_\e)\leq C\e^{n/2}.$$
\end{theorem}

By taking $u_{\e}=T^\e(v_{\e})$ above, we clearly have the following.

\begin{corollary}\label{Cor:u epsilon}
There exists at least one positive {nonconstant} solution {$u_\e\in C^{1,\alpha}(\overline{\Omega})$,
for $0<\alpha<1$,}
to problem \eqref{eq.8}, {provided $\e>0$ is sufficiently small}.
\end{corollary}

The rest of this section is devoted to the proof of Theorem \ref{TeoremaMountain}.
We split the proof in several steps.

\medskip

\noindent $\bullet$ {\sl Proof of existence of a {nonconstant} critical point {for $\e>0$ sufficiently small}.}
We apply the Mountain Pass Lemma by Ambrosetti--Rabinowitz
\cite[Theorem 2.1]{AMBRRAB}, see also \cite{Lin-Ni-Takagi} and
\cite[Chapter~8]{Evans}, to find a {nonconstant} critical point $v_\e$ of the functional $\mathcal{I}_\e$, over
the Hilbert space $\mathsf{H}^\e(\mathcal{C})$. Observe that
$\mathcal{I}_\e(0)=0$.
The application of the Mountain Pass Lemma will give us a nonconstant nonzero solution $v_\e$. To this aim
we check several points.

\smallskip

\noindent\texttt{1.} \emph{The functional $\mathcal{I}_\e$ is {in} $C^1(\mathsf{H}^\e(\mathcal{C});\R)$ and
$\mathcal{I}'_\e$ is Lipschitz continuous on bounded sets of $\mathsf{H}^\e(\mathcal{C})$.}
Indeed, we have
$$\mathcal{I}_\e(v)=\frac{1}{2}\|v\|_{\e}^2-\int_{\Omega}G(T^{\e} v)\,dx=
\mathcal{I}_{\e,1}(v)-\mathcal{I}_{\e,2}(v),$$
where
\[
\mathcal{I}_{\e,1}(v):=\frac{1}{2}\|v\|_{\e}^2
\]
and
\[
\mathcal{I}_{\e,2}(v):=\int_{\Omega}G(T^{\e} v)\,dx.
\]
It is standard to see (see for example \cite{Evans})  that $\mathcal{I}_{\e,1}$ satisfies the conditions in \texttt{1} above and, moreover,
$\mathcal{I}_{\e,1}'(v)=v$. We must now analyze $\mathcal{I}_{\e,2}$. To this end let us consider the
operator $\mathcal{K}$ that maps every $f\in H^{-1/2}(\Omega)$ into the weak solution $v\in \mathsf{H}^\e(\C)$
of the linear problem \eqref{eq.6}. Then we have that
\[
\mathcal{K}:H^{-1/2}(\Omega)\rightarrow {\mathsf{H}^{\e}(\C)}
\]
is an isometry. We can check
that for any $v\in \mathsf{H}^\e(\C)$, $\mathcal{I}_{\e,2}'(v)=\mathcal{K}[g(T^{\e} v)]$.
In fact, using parallel arguments as in \cite[Chapter~8]{Evans}, we have that for any
$w\in \mathsf{H}^\e(\C)$
\begin{align*}
\mathcal{I}_{\e,2}(w) &= \int_\Omega G(T^{\e} v)\,dx+\int_\Omega g(T^{\e} v)
\big(T^{\e}(w-v)\big)dx+R \\
&= \mathcal{I}_{\e,2}(v)+\left(\mathcal{K}[g(T^{\e}v)],w-v\right)_{\e}+R,
\end{align*}
where we used \eqref{eq.33} and $R$ is some remainder. As in \cite[p.~484]{Evans} again, by
applying the trace
inequality \eqref{traceeqcomp} we can conclude that $R=o(\|w-v\|_{\e})$.
The Lipschitz continuity of $\mathcal{I}^{\prime}_{{\e,2}}$ on bounded sets follows similarly.
Thus $\mathcal{I}_{\e,2}$ satisfies condition \texttt{1} above. Therefore
$$\mathcal{I}_\e'(v)=v-\mathcal{K}[g(T^{\e} v)].$$

\noindent\texttt{2.} \emph{The functional $\mathcal{I}_\e$ satisfies the Palais--Smale condition.}
We have to show that if we choose a sequence $\{v_k\}_{k\in\N}$ in $\mathsf{H}^\e(\C)$
such that $\{\mathcal{I}_\e(v_k)\}_{k\in\N}$ is bounded
and
\begin{equation}
\mathcal{I}_\e'(v_k)=v_k-\mathcal{K}[g(T^{\e} v_k)]\to 0\label{convk}
\end{equation} as $k\to\infty$
in $\mathsf{H}^\e(\mathcal{C})$, then $\{v_k\}_{k\in\N}$ is precompact in $\mathsf{H}^\e(\C)$.
To that end, let {$\eta>0$}. We have
$$|(\mathcal{I}_\e'(v_k),w)_{{\e}}|\leq {\eta}\|w\|_{ {\e}},$$
for $k$ large enough. If we choose $w=v_k$ then, by \eqref{eq.33},
$$\left|\|v_k\|_{\e}^2-\int_\Omega g(T^{\e} v_k)T^{\e}v_k\,dx\right|
\leq  {\eta}\|v_k\|_{ {\e}},$$
for $k$ large. In particular, for $ {\eta}=1$,
$$\int_\Omega g(T^{\e} v_k) {T^\e}v_k\,dx\leq \|v_k\|^2_{\e}+\|v_k\|_{\e}.$$
 {Since $|\mathcal{I}_\e(v_k)|\leq C$ for all $k$ and $g(t)=t^p$ for $t>0$ and $p>1$, we deduce
$$\|v_k\|_\e^2 \leq C+2\int_\Omega G(T^\e v_k)\,dx \leq C+\frac{2}{p+1}(\|v_k\|_\e^2+
\|v_k\|_\e).$$
Now $2/(p+1)<1$, therefore $\{v_k\}_{k=1}^\infty$}
 is bounded in $\mathsf{H}^\e(\C)$. Then, up to subsequences, we have
\[
v_{k}\rightharpoonup v,\quad \hbox{weakly in}~\mathsf{H}^{\e}(\mathcal{C}).
\]
By Corollary \ref{compactness of the trace} we have that
\[
T^{\e}v_{k}\rightarrow T^{\e}v,\quad\hbox{strongly in}~L^{p+1}(\Omega).
\]
Notice again here that $p+1<2n/(n-1)$.
Then we find $g(T^{\e}v_{k})\rightarrow g(T^{\e}v)$ in $H^{-1/2}(\Omega)$, thus $\mathcal{K}[g(T^{\e} v_k)]\rightarrow
\mathcal{K}[g(T^{\e} v)]$ strongly in $\mathsf{H}^{\e}(\mathcal{C})$. But then by \eqref{convk} we conclude that
\[
v_{k}=v_{k}-\mathcal{K}[g(T^{\e} v_k)]+\mathcal{K}[g(T^{\e} v_k)]\rightarrow\mathcal{K}[g(T^{\e} v)]
\]
in $\mathsf{H}^{\e}(\mathcal{C})$. Thus $v=\mathcal{K}[g(T^{\e} v)]$ and condition \texttt{2} holds.
\smallskip

\noindent\texttt{3.} \emph{The functional $\mathcal{I}_{\e}$ satisfies the following condition: there is some $\rho>0$ such that $\mathcal{I}_{\e}(v)>0$ for
$0<\|v\|_{\e}<\rho$, and $\mathcal{I}_{\e}(v)\geq\beta>0$ for some $\beta>0$ and $\|v\|_{\e}=\rho$.}
This follows from \cite[Lemma 3.3]{AMBRRAB}. Indeed, it suffices to show that
\[
\int_{\Omega}G(T^{\e} v)\,dx=o(\|v\|_{\e}^{2}),
\]
which is readily true because the trace Sobolev inequality (recall that $1<p<(n+1)/(n-1))$) yields
\[
\left|\int_{\Omega}G(T^{\e} v)\,dx\right|\leq C \|v\|_{\e}^{p+1}.
\]
\smallskip
\noindent\texttt{4.} \emph{For a sufficiently small $\e>0$, there is a nonegative function $\Phi\in
\mathsf{H}^{\e}(\mathcal{C})$ and positive constants $t_{0}$, $C_{0}$ such that
\[
\mathcal{I_{\e}}(t_{0}\Phi)=0,
\]
and
\[
\mathcal{I_{\e}}(t\Phi)\leq C_{0} \e^{n/2},\quad\hbox{for}~t\in[0,t_{0}].
\]}
Observe that $\|t_0\Phi\|_\e>\rho$, where $\rho$ is as in \texttt{3} above.
Indeed, we can choose $\Phi$ as
\[
\Phi(x,y)=e^{-y/2}\varphi(x),
\]
where $\varphi$ is as in \cite[p.~10]{Lin-Ni-Takagi}
\begin{equation*}
\varphi(x)=
\begin{cases}
 \e^{-n/2}(1-\e^{-1/2}|x|),&\hbox{if}~|x|<\sqrt{\e},\\
 0 &\hbox{if}~|x|\geq\sqrt{\e}.
\end{cases}
\end{equation*}
We can also suppose that
$0\in \Omega$ and that
$\e$ is sufficiently small so that $\Phi\in\mathsf{H}^{\e}(\mathcal{C})$. Of course we have
\[
\iint_{\mathcal{C}}|\nabla_{x}\Phi|^2\,dx\,dy=\int_{\Omega}|\nabla\varphi|^2\,dx,\quad
\iint_{\mathcal{C}}|\Phi_{y}|^2\,dx\,dy=\frac{1}{4}\int_{\Omega}\varphi^2\,dx,
\]
and clearly $\tr_{\Omega}\Phi(x)=\varphi(x)$, so that
\[
\int_{\Omega}|T^{\e}\Phi|^2 \,dx =\int_{\Omega}\varphi^2\,dx.
\]
Then, by following the same arguments as in \cite[Lemma 2.4]{Lin-Ni-Takagi} if we set
\[
\mathsf{g}(t)=\mathcal{I}_{\e}(t\Phi),\quad\hbox{for}~t\geq0,
\]
it is seen that there exist $t_{1},\,t_2$ with $0<t_1<t_2$ such that $\mathsf{g}^{\prime}(t)<0$ if $t>t_1$ and $\mathsf{g}(t)<0$ if $t>t_2$. The details of this proof are left to the interested reader. Now property
\texttt{3}
implies $\mathsf{g}(t)>0$ for small $t$, thus there is $t_0$ such that $\mathsf{g}(t_{0})=0$, that is
\[
\mathcal{I}_{\e}(t_{0}\Phi)=0,
\]
for small $\e>0$. Thus  {we get that} $\|t_0\Phi\|_\e>\rho$, where $\rho$ is as in \texttt{3} above. Moreover, as in \cite[p.~12]{Lin-Ni-Takagi}, we have
\begin{equation}
\max_{t\geq0} \mathsf{g}(t)=\max_{0\leq t\leq t_1}\mathsf{g}(t)\leq C_{0} \e^{n/2}\label{eq.202}
\end{equation}
for some constant $C_{0}>0$. Hence \texttt{4} is proved.

\smallskip

\noindent\texttt{5.} \emph{Conclusion.} We are in position to apply the Mountain Pass Lemma.
Set $E=\mathsf{H}^{\e}(\mathcal{C})$, $e=t_{0}\Phi$ and
\[
\Gamma=\left\{\gamma\in C([0,1];E):\gamma(0)=0,\gamma(1)=e\right\}.
\]
Since conditions \texttt{1}--\texttt{4} are satisfied, the Mountain Pass Lemma implies that the number
\[
c=\min_{\gamma\in\Gamma}\max_{t\in [0,1]}\mathcal{I}_{\e}(\gamma(t))
\]
is a critical value of $\mathcal{I}_{\e}$ in $\mathsf{H}^{\e}(\mathcal{C})$. Thus there exists $v_{\e}$
in $\mathsf{H}^{\e}(\mathcal{C})$ such that
\[
\mathcal{I}^{\prime}_{\e}(v_{\e})=0
\]
and, by \eqref{eq.202},
\[
\mathcal{I}_{\e}(v_{\e})=c\leq \max_{[0,t_{0}]}\mathcal{I}_{\e}(t\Phi)\leq C_{0} \e^{n/2}.
\]
It remains to prove that $v_{\e}$ is nonconstant. Let us argue by contradiction.
Suppose that $v_\e=c_1$, for some real number $c_{1}$. Then
$$\mathcal{I}_{\e}(v_{\e})=\left(\frac{1}{2}c^2_{1}-G(c_1)\right)|\Omega|.$$
Since $v_{\e}$ is a critical point, we have $\mathcal{I}_\e'(v_e)=0$,
which by using the equation implies that $g(c_1)=c_1$, and therefore $c_1=1$. So,
$$\mathcal{I}_{\e}(v_{\e})=\left(\frac{1}{2}-\frac{1}{p+1}\right)|\Omega|.$$
This is in contradiction with the inequality $\mathcal{I}_{\e}(v_{\e})\leq C \e^{n/2}$,
that holds for small $\e$.

To summarize, we conclude that for small $\e$ the functional $\mathcal{I}_{\e}(v_{\e})$ has at least one nonzero nonconstant critical point.\qed

\medskip

\noindent $\bullet$ {\sl Proof of smoothness for each $\e>0$ small.} We prove now that the nonconstant minimizers $v_\e$ we found in the first part of the proof are actually  {classical solutions}.
To this aim, let $u\in H^{1/2}(\Omega)$ be a solution to
 $$(-\e\Delta_N)^{1/2}u+u=g(u),$$
 for
 $$1<p<\frac{n+1}{n-1},$$
 and let $v$ be its $\e$--Neumann extension, solving problem \eqref{eq.10}. If we know that $g(u)\in L^\infty(\Omega)$, then Theorem \ref{Thm:regularity} parts (3)--(4) imply that $u\in C^{1,\alpha}(\overline{\Omega})$.
  {Arguing as in the proof of Theorem \ref{Thm:regularity} part (4),
 we can see that $v\in C^{2,\alpha}(\mathcal{C})\cap C^{1,\alpha}(\overline{\mathcal{C}})$.
 Indeed, take $h=\e^{-1/2}(g(u)-u)\in C^{1,\alpha}(\overline{\Omega})$ and notice that,
 by interior Schauder estimates, the solution $w$ to \eqref{w equation} is in $C^{3,\alpha}(\mathcal{C})$.}
 Hence each $v_\e$ is a classical solution.  {Therefore} we are reduced to prove that $u$ is bounded.
 We use a bootstrap argument in $u$ with the aid of Theorem \ref{Thm:regularity}.
 To that end, recall the embedding $H^{1/2}(\Omega)\subset L^{\frac{2n}{n-1}}(\Omega)$.
 This gives that
 $$u\in L^q(\Omega),\quad\hbox{for}~q=\frac{2n}{n-1}>2.$$
 Then, since $n\geq2$, it is clear that $q>p$.

 Let us suppose first that $n\geq3$. Then, since $p>1$, we have $p<q<np$.
  Observe that, by the condition on $p$,
  $$n(p-1)<n\left(\frac{n+1}{n-1}-1\right)=\frac{2n}{n-1}=q.$$
  Then
  $$\theta:=\frac{n}{np-q}>1.$$
  Suppose now that $u\in L^r(\Omega)$ for some $q\leq r<np$.
  Then for the nonlinear term we have $g(u)\in L^{r/p}(\Omega)$. Since $r/p\geq q/p>1$,
  and $r/p<n$, we have that $(-\e\Delta_N)^{1/2}u+u\in L^{\gamma}(\Omega)$ ($\gamma=r/p$)
  with $1<\gamma<n$. Hence, from  {\eqref{reg2} in}
  Theorem \ref{Thm:regularity}, we find $u\in L^{\theta r}(\Omega)$
  (observe that $\theta r$ is certainly bigger than $\gamma$ and smaller
  than $\frac{n\gamma}{n-\gamma}$, the latter because $q\leq r$). Now we iterate this procedure
  in the following way. Choose a positive integer $k$ for which $\theta^k q< np<\theta^{k+1}q$.
  Then repeat the same reasoning as above but choosing $r=\theta^j q$, for $j=0,1,\ldots,k$.
  At the end one deduces that $u\in L^{\theta^{k+1}q}(\Omega)$. The fact that
  the nonlinear term $g(u)$ is in $L^{\theta^{k+1}q/p}(\Omega)$ and that such exponent is
  strictly bigger than $n$ imply, by  {\eqref{reg3} in} Theorem \ref{Thm:regularity}, that $u$ is bounded.\\
 Next we assume that $n=2$, so $q=4$ and $1<p<3$. We consider now three possible cases.\smallskip\\
 \textbf{Case I. $p<2$.} Then $g(u)\in L^{4/p}(\Omega)$ and $4/p>2=n$. This
  says, by Theorem \ref{Thm:regularity}$(c_2)$, that $(-\e\Delta_N)^{1/2}u+u\in L^{r}(\Omega)$ for some $r>n$. By \eqref{reg3} we obtain that
  $u\in L^\infty(\Omega)$.\\
 \textbf{Case II. $p=2$.} For the right hand side we have $|u|^{p-1}u=u^2\in L^2(\Omega)$.
  Then, by Theorem \ref{Thm:regularity}$(c_2)$, $u$ is in $L^r(\Omega)$ for all $r\geq 2$, so \eqref{reg3}
  gives $u\in L^\infty(\Omega)$.\smallskip\\
  \textbf{Case III. $2<p<3$.} Here we have $p<q=4<2p=np$, so we can apply the iteration
  as in the case $n\geq3$ above to get higher integrability for the right hand side that still ensures
  the boundedness of $u$.\qed

  \medskip

\noindent $\bullet$ {\sl Proof of positivity for each $\e>0$ small}.
From the bootstrap argument we have proved that  {$v_{\e}\in
C^{2,\alpha}(\mathcal{C})\cap C^{1,\alpha}(\overline{\mathcal{C}})$}
for any $0<\alpha<1$, where $v_{\e}$ is a nonconstant critical
point of the functional $\mathcal{I}_{\e}$ in \eqref{eq.11}. In order to prove that $v_\e>0$ everywhere in $\overline{\mathcal{C}}$,  {let us choose $v^{-}_\e$ in the weak formulation \eqref{weakH} of problem \eqref{eq.10}. Then we have
\[
\iint_{\mathcal{C}}\big(\e|\nabla_{x}v_{\e}^{-}|^{2}+|(v_{\e}^{-})_{y}|^{2}\big)\,dx\,dy
+\int_{\Omega}|u_{\e}^{-}|^{2}\,dx=-\int_{\Omega}(u^{+}_{\e})^{p}\,u_{\e}^{-}
\,dx=0.
\]
Thus $v_\e\geq0$ in $\mathcal{C}$ and $u_\e\geq0$ in $\Omega$.
Then it suffices to use \cite[Proposition~7]{PellacciMontef} and
\cite[Remark~5]{PellacciMontef} to get
 $u_\e>0$ in $\overline{\Omega}$ and $v_\e>0$ in $\overline{\mathcal{C}}$}.\qed

%%%%%%%%%%%%%%%%%%%%%%%%%%%%%%%%%%%%%%%%%%%%%%%%%%%%%%
\section{Boundedness, spike shape of solutions and nonexistence for large $\e$}\label{Section5}
%%%%%%%%%%%%%%%%%%%%%%%%%%%%%%%%%%%%%%%%%%%%%%%%%%%%%%

%%%%%%%%%%%%%%%%%%%%%%%%%%%%%%%%%%%%%%%%%%%%%%%%%%%%%%
\subsection{Uniform boundedness for small $\e$}
%%%%%%%%%%%%%%%%%%%%%%%%%%%%%%%%%%%%%%%%%%%%%%%%%%%%%%

We have shown so far that each solution $u_\e$ to problem \eqref{eq.8} is bounded for small $\e$.
The next result proves that the family of solutions $\left\{u_{\e}\right\}_{\e>0}$ (for small $\e$) is,
in fact, equibounded.

\begin{theorem}\label{Moser}
Let $v_{\e},u_\e$ be the  {nonconstant} smooth positive solutions obtained by Theorem \ref{TeoremaMountain}
and Corollary \ref{Cor:u epsilon}. Then
\begin{equation}
\e\iint_{\mathcal{C}}|\nabla_{x} v_{\e}|^{2}\,dx\,dy+\iint_{\mathcal{C}} |(v_{\e})_{y}|^{2}\,dx\,dy+\int_{\Omega}|u_{\e}|^{2}\,dx=\int_{\Omega}|u_{\e}|^{p+1}\,dx\leq (2^{-1}-\theta)^{-1}C\e^{n/2},\label{firstestim}
\end{equation}
where $C$ is the constant of Theorem \ref{TeoremaMountain} and $\theta=1/(p+1)$.
In particular, $u_\e\to0$ in measure in $\Omega$ as $\e\to0^+$.
Moreover, there is a constant $C_{1}>0$, depending on $\Omega$ and $C$, such that
\begin{equation}
\sup_{\Omega}u_{\e}\leq C_{1}.\label{equibound}
\end{equation}
\end{theorem}

\begin{proof}
The proof employs a suitable adaptation of the arguments of \cite[Corollary~2.1]{Lin-Ni-Takagi}. First observe that by taking $v_{\e}$ in the weak formulation of \eqref{eq.10} we find
\begin{equation}
\e\iint_{\mathcal{C}}|\nabla_{x} v_{\e}|^{2}\,dx\,dy+\iint_{\mathcal{C}} |(v_{\e})_{y}|^{2}\,dx\,dy+\int_{\Omega}|v_{\e}(x,0)|^{2}\,dx=\int_{\Omega}|v_{\e}(x,0))|^{p+1}\,dx,\label{eq.22}
\end{equation}
so that
\[
\mathcal{I}_{\e}(v_{\e})=\frac{1}{2}\int_{\Omega}|v_{\e}(x,0))|^{p+1}\,dx-\frac{1}{p+1}\int_{\Omega}|v_{\e}(x,0)|^{p+1}\,dx=\left(\frac{1}{2}-\theta\right)\int_{\Omega}|v_{\e}(x,0)|^{p+1}\,dx,
\]
where
\[
\theta=\frac{1}{p+1}<\frac{1}{2}.
\]
Then Theorem \ref{TeoremaMountain} implies
$$\int_{\Omega}|v_{\e}(x,0)|^{p+1}\,dx\leq (2^{-1}-\theta)^{-1}C\e^{n/2}.$$
From this and \eqref{eq.22} we get \eqref{firstestim}.

To prove the uniform boundedness for small $\e$ we apply a classical Moser iteration, see \cite{Gilbarg-Trudinger, Lin-Ni-Takagi}. Let us choose in the weak formulation of problem \eqref{eq.10},
see Definition \ref{def:weak},  the test function
\[
\psi=v_{\e}^{2s-1},
\]
for some $s\geq1$. Then we have
\begin{equation}
\begin{aligned}
&\e\frac{2s-1}{s^2}\iint_{\mathcal{C}}|\nabla_{x} (v^{s}_{\e})|^{2}\,dx\,dy+\frac{2s-1}{s^2}\iint_{\mathcal{C}} |(v^{s}_{\e})_{y}|^{2}\,dx\,dy+\int_{\Omega}|v_{\e}^{s}(x,0)|^{2}\,dx\label{eq.18}\\&=\int_{\Omega}|v_{\e}(x,0)|^{p-1+2s}\,dx.
\end{aligned}
\end{equation}
Now, for $\e<1$, using the trace inequality \eqref{traceeqcomp} (where $\mathsf{C}(\e)=c\e^{1/4}$) and the fact that $(2s-1)/s^{2}<1$, we find from \eqref{eq.18}
\begin{equation*}
C_{0}^{2}c^2\e^{1/2}\frac{2s-1}{s^2}\left(\int_{\Omega}|v_{\e}(x,0)|^{s\nu}\,dx\right)^{2/\nu}\leq\int_{\Omega}|v_{\e}(x,0)|^{p-1+2s}\,dx,
\end{equation*}
where $\nu$ is the Sobolev trace embedding exponent
\[
\nu=\frac{2n}{n-1}.
\]
Since $(2s-1)s^{-2}\geq s^{-1}$ we get
\begin{equation}
\left(\int_{\Omega}|v_{\e}(x,0)|^{s\nu}\,dx\right)^{2/\nu}\leq\e^{-1/2}s\gamma^{2}\int_{\Omega}|v_{\e}(x,0)|^{p-1+2s}\,dx\label{eq.19}
\end{equation}
where $\gamma=(C_{0}c)^{-1}$.
Parallel to \cite[p.~14]{Lin-Ni-Takagi}, we define two sequences $\left\{s_{j}\right\}_{j=0}^\infty$
 and $\left\{M_{j}\right\}_{j=0}^\infty$ by setting
\[
p-1+2s_{0}=\nu,\quad p-1+2s_{j+1}=\nu s_{j},
\]
and
\[
M_{0}=((2^{-1}-\theta)^{-1}\gamma^{2} C)^{\nu/2},\quad M_{j+1}=(\gamma^{2}s_{j}M_{j})^{\nu/2}.
\]
In particular, we have that $s_{j}>1$ and $s_{j}\rightarrow\infty$ as $j\rightarrow\infty$.
We want to show that
\begin{equation}
\int_{\Omega}|v_{\e}(x,0)|^{p-1+2s_{j}}\,dx\leq M_{j}\e^{n/2},\label{eq.23}
\end{equation}
and
\begin{equation}
M_{j}\leq e^{ms_{j-1}},\label{eq.24}
\end{equation}
for some $m>0$.
Let us prove \eqref{eq.23} for $j=0$. Using \eqref{traceeqcomp} and \eqref{firstestim} one finds
\begin{align*}
\int_{\Omega}|v_{\e}(x,0)|^{p-1+2s_{0}}\,dx&=\int_{\Omega}|v_{\e}(x,0)|^{\nu}\,dx\leq \e^{-\nu/4}C^{-\nu}_{0}\|v_{\e}\|^{\nu}_{{\e}}\\
&\leq C_{0}^{-\nu}(2^{-1}-\theta)^{-\nu/2}C^{\nu/2}\,\e^{\nu n/4-\nu/4}\\
&=\gamma^{\nu}(2^{-1}-\theta)^{-\nu/2}C^{\nu/2}\,\e^{\nu n/4-\nu/4}\\
&=M_{0}\e^{\nu n/4-\nu/4}=M_0\e^{n/2}.
\end{align*}
Furthermore, using \eqref{eq.19} it is not difficult to show that if \eqref{eq.23} holds for $j>0$, then it holds for $j+1$ too. Also \eqref{eq.24} follows from the proof of \cite[Corollary~2.1]{Lin-Ni-Takagi}. Finally, by applying \eqref{eq.19}--\eqref{eq.23}--\eqref{eq.24},
\begin{align*}
\left(\int_{\Omega}|v_{\e}(x,0)|^{s_{j-1}\,\nu}\,dx\right)^{\frac{1}{\nu s_{j-1}}}&\leq\e^{-\frac{1}{4s_{j-1}}}(\gamma^{2})^{\frac{1}{2s_{j-1}}}s_{j-1}^{\frac{1}{2s_{j-1}}}
\left(\int_{\Omega}|v_{\e}(x,0)|^{p-1+2s_{j-1}}\,dx\right)^{\frac{1}{2s_{j-1}}}\\
&\leq\e^{\frac{n-1}{4s_{j-1}}}(\gamma^{2})^{\frac{1}{2s_{j-1}}}s_{j-1}^{\frac{1}{2s_{j-1}}}e^{m/2}.
\end{align*}
By letting $j\rightarrow\infty$, inequality \eqref{equibound} follows.
\end{proof}

%%%%%%%%%%%%%%%%%%%%%%%%%%%%%%%%%%%%%%%%%%%%%%%%%%%%%%
\subsection{Shape of solutions}
%%%%%%%%%%%%%%%%%%%%%%%%%%%%%%%%%%%%%%%%%%%%%%%%%%%%%%

We show that for small $\e$, the solution $u_{\e}$ of our fractional semilinear problem given by Corollary \ref{Cor:u epsilon} concentrates around some points and its graph looks like spikes on $\overline{\Omega}$.

\begin{theorem}[Shape of $u_\e$]\label{Shape}
For $K=(k_1\,\ldots,k_{n})\in\mathbb{Z}^{n}$ and $l>0$, define the cube of $\R^{n}$
\[
Q_{K,l}=\left\{(x_{1},\ldots,x_{n})\in\R^n:|x_{i}-lk_{i}|\leq\frac{l}{2},\,1\leq i\leq n\right\}.
\]
For small $\e$, let us consider the solution $u_\e$ given by Corollary \ref{Cor:u epsilon}
 and define for all $\eta>0$ the upper level set of $u_\e$
\[
\Omega_{\eta}=\left\{x\in\Omega:u_{\e}(x)>\eta\right\}.
\]
Then there is a positive integer $m$ depending only on $\Omega$, the constant $C$ appearing in Theorem \ref{TeoremaMountain} and $\eta$, such that $\Omega_\eta$ is covered by at most $m$ of
the $Q_{K,\sqrt{\e}}$ cubes.
\end{theorem}

By checking the proof of \cite[Proposition 4.1]{Lin-Ni-Takagi}, we see that
in order to prove Theorem \ref{Shape}
we only need the Harnack inequality of Theorem \ref{Harnack} and Lemma \ref{Lemma 4.1}.
The latter result is a consequence of the following proposition and its
proof follows the lines of \cite[Lemma~2.3]{Lin-Ni-Takagi}. Finally, the proof of Proposition \ref{contradarg}
uses Theorem \ref{Moser} with slight modifications of the arguments
in the proof of \cite[Proposition 2.2]{Lin-Ni-Takagi}. The rather cumbersome details are left to the interested reader.

\begin{proposition}\label{contradarg}
Fix $\e_0>0$. Then there is a constant $c_{0}>0$ such that
\[
\e\iint_{\mathcal{C}}|\nabla_{x} v_{\e}|^{2}\,dx\,dy+\iint_{\mathcal{C}} |(v_{\e})_{y}|^{2}\,dx\,dy+\int_{\Omega}|u_{\e}|^{2}\,dx\geq c_{0}\e^{n/2},
\]
for all $\e\in(0,\e_{0})$ and any solution $v_{\e}$ to \eqref{eq.10} whose trace is $v_{\e}(\cdot,0)=u_{\e}$, which solves \eqref{eq.8}.
\end{proposition}

\begin{lemma}\label{Lemma 4.1}
Let $u_\e$ be as in Corollary \ref{Cor:u epsilon}. Then
$$m(q)\e^{n/2}\leq\int_{\Omega}u_{\e}^{q}\,dx\leq M(q)\e^{n/2},\quad\hbox{if}~1\leq q<+\infty,$$
$$m(q)\e^{n/2}\leq\int_{\Omega}u_{\e}^{q}\,dx\leq M(q)\e^{nq/2},\quad\hbox{if}~0<q<1,$$
where $m(q),M(q)$ are positive constants independent of $\e$, such that $m(q)<M(q)$.
\end{lemma}

%%%%%%%%%%%%%%%%%%%%%%%%%%%%%%%%%%%%%%%%%%%%%%%%%%%%%%
\subsection{Uniform boundedness for all $\e>0$}
%%%%%%%%%%%%%%%%%%%%%%%%%%%%%%%%%%%%%%%%%%%%%%%%%%%%%%

We have shown in Theorem \ref{Moser} that the solutions $u_{\e}$ determined in
Corollary \ref{Cor:u epsilon} are uniformly bounded. The following result shows that
this uniform boundedness property can be extended to \emph{all} $\e$, no matter how small they are.

\begin{theorem}[Uniform boundedness in $\e>0$]\label{Unifboundtheo}
 There exists a positive constant $C$ independent of  $\e$ such that for any positive solution $u$
 to \eqref{eq.8} we have
 $$\sup_\Omega u\leq C.$$
\end{theorem}

\begin{proof}
 The proof is based on a combination of techniques that are parallel to the
 arguments used in \cite[Theorem~3(i)]{Lin-Ni-Takagi}, whose roots can be tracked down to
 one of the famous papers by B. Gidas and J. Spruck \cite{Gidas-Spruck}.
We need to flatten the boundary of $\Omega$ and then use a blow up technique together
 with a Liouville result for the fractional Neumann Laplacian.

 The proof is divided in two steps. First one shows that, for a fixed $\e_0>0$, the estimate holds
 for all solutions of \eqref{eq.8} uniformly in $0<\e\leq\e_0$. The second step is
 to give the proof when $\e\geq\e_0$.

 \

 \noindent\texttt{Step 1.} Let $\e_0>0$ be fixed. The proof goes by contradiction. That is, suppose
 there exists a sequence of positive solutions $\{u_k\}_{k\in\N}$ of \eqref{eq.8} corresponding
 to parameters $\{\e_k\}_{k\in\N}$, with $0<\e_k\leq\e_0$, and a sequence of points $P_k\in\overline{\Omega}$
 such that
 $$M_k:=\sup_\Omega u_k=u_k(P_k)\to\infty,\quad\hbox{and}\quad P_k\to P\in\overline{\Omega},\quad\hbox{as}~k\to\infty.$$
 Let $v_k(x,y)$ be the solution to the extension problem \eqref{eq.10} corresponding to each $u_k$,
 therefore $v_k(x,0)=u_k(x)$. By Hopf's maximum principle, the maximum of $v_{k}$ can lie only on $\overline{\Omega}\times \left\{0\right\}$, thus $\sup_{\overline{\mathcal{C}}} v_{k}=v_{k}(P_{k},0)=M_{k}$. In this step we have two cases, depending on where $P$ lies.

 \noindent\texttt{Case 1.}  Suppose that $P\in\partial\Omega$.
 Without loss of generality we can assume that $P$ is the origin and that
 the exterior normal to $\partial\Omega$ at $P$ is the vector $-e_n\in\R^n$.
 Arguing as in the proof of Theorem \ref{Thm:regularity}(4), we can straighten the boundary near $P$ with a
  local diffeomorphism $\Psi$. Let us call $z$ the new coordinates, and let
 $\widetilde{v}_k(z,y)=v_k(x,y)$. As in \cite[p.~19]{Lin-Ni-Takagi}
 it can be verified that $\widetilde{v}_k(z,y)$ satisfies the following extension problem:
 \begin{equation}\label{ext}
 \begin{cases}
  \e_kL_z\widetilde{v}_k+(\widetilde{v}_k)_{yy}=0,&\hbox{in}~(B_{2\delta}\cap\{z_n>0\})\times(0,\infty),\\
  \partial_{\nu}\widetilde{v}_k=0,&\hbox{on}~(B_{2\delta}\cap\{z_n=0\})\times[0,\infty),\\
  -(\widetilde{v}_k)_y(x,0)=
  g(\widetilde{v}_k(x,0))-\widetilde{v}_k(x,0),&\hbox{on}~B_{2\delta}\cap\{z_n>0\},
 \end{cases}
 \end{equation}
for some small $\delta>0$. Here $L_z$ is a nondivergence form elliptic operator with
smooth coefficients and no independent term
acting in the $z$-variable only and $B_{2\delta}$ is the ball on $\R^n$ centered at the
origin with radius $2\delta$.
Let $Q_k=\Psi(P_k)=(q_k',\alpha_k)$, $\alpha_k\geq0$. Since $Q_k\to0$,
we can assume that $|Q_k|<\delta$. Notice that
$$\lambda_k:=\Big(\frac{\e_k}{M_k^{p-1}}\Big)^{1/2}\to0,\quad\hbox{as}~k\to\infty.$$
We have now two subcases.

\noindent\texttt{Subcase 1.1.} $\alpha_k/\lambda_k$ remains bounded as $k\to\infty$. Then, up to a subsequence,
$\alpha_k/\lambda_k\to\alpha\geq0$. Define then the rescaled function
$$w_k(z,y):=\frac{1}{M_k}\widetilde{v}_k(\lambda_k z'+q_k',\lambda_kz_n,y),\quad z=(z',z_n)\in (B_{\delta/\lambda_k}
\cap\{z_n>0\}),~y>0.$$
Observe that $0<w_k\leq1$.
Then, from \eqref{ext} we can verify that $w_k$ satisfies the extension problem
$$\begin{cases}
  \widetilde{L}^k_zw_k+(w_k)_{yy}=0,&\hbox{in}~(B_{\delta/\lambda_k}\cap\{z_n>0\})\times(0,\infty),\\
  \partial_{\nu}w_k=0,&\hbox{on}~(B_{\delta/\lambda_k}\cap\{z_n=0\})\times[0,\infty),\\
  -(w_k)_y(x,0)=g(w_k(x,0))-M_k^{-(p-1)}w_k(x,0),
  &\hbox{on}~B_{\delta/\lambda_k}\cap\{z_n>0\}.
 \end{cases}$$
The coefficients of $\widetilde{L}^k_z$ are now a rescaled and translated version of the coefficients
of $L_z$ and are uniformly bounded in $k$. Then the compactness
arguments in \cite{Lin-Ni-Takagi} can be paralleled in such a way that
we can extract from $\left\{w_{k}\right\}_{k\in\N}$ a subsequence converging uniformly to a nonnegative solution $w(z,y)$
to the extension problem
$$\begin{cases}
  \Delta_zw+w_{yy}=0,&\hbox{in}~\R^n_+\times(0,\infty),\\
  \partial_{\nu}w=0,&\hbox{on}~\partial\R^n_+\times[0,\infty),\\
  -w_y(x,0)=g(w(x,0)),&\hbox{on}~\R^n_+.
 \end{cases}$$
 Let us now extend $w$ to $\R^{n}\times[0,\infty)$ as $w^\ast(z',z_n,y)=w(z',|z_n|,y)$,
 so that $w^\ast$ is a solution to
$$\begin{cases}
  \Delta_zw^\ast+w^\ast_{yy}=0,&\hbox{in}~\R^n\times(0,\infty),\\
  -w^\ast_y(x,0)=(w^\ast(x,0))^p,&\hbox{in}~\R^n.
 \end{cases}$$
The Liouville theorem of \cite{Hu} (see also \cite[Remark~1.4]{YanYanLi}) implies that
 $w^{\ast}$  is identically zero. But this a contradiction because
$$w(0,\ldots,0,\alpha,0)=\lim_{k\to\infty}w_k(0,\ldots,0,\alpha_k/\lambda_k,0)
=\lim_{k\to\infty}\frac{1}{M_k}\widetilde{v}_k(Q_k,0)=1.$$

 \noindent\texttt{Subcase 1.2.} $\alpha_k/\lambda_k$ is unbounded. We can suppose that $\alpha_k/\lambda_k\to\infty$.
 Now we define
 $$w_k(z,y)=\frac{1}{M_k}\widetilde{v}_k(\lambda_kz+Q_k,y).$$
 Then the argument goes as in \cite{Lin-Ni-Takagi} with the proper modifications
 in the extension problem as we did in subcase 1.1, and using the Liouville theorem of \cite{Hu,YanYanLi}.

 \noindent\texttt{Case 2.} Suppose that $P$ is in the interior of $\Omega$.
 The scaling we need now is
 $$w_k(z,y)=\frac{1}{M_k}\widetilde{v}_k(\lambda_kz+P_k,y),$$
 with $\lambda_k$ as above and the argument goes as in as in Subcase 1.2. Details are left
 to the interested reader.

 \

 \noindent\texttt{Step 2.}  If $\e\geq\e_{0}$, by arguing by contradiction to fall into Step 1, we can prove that
 \begin{equation}
 \sup_{\Omega} u\leq C \e^{1/2(p-1)}\label{eq.27}
 \end{equation}
 being $C$ a constant independent on $\e$. From \eqref{eq.18} and \eqref{eq.27} we obtain
 \begin{equation}
\begin{aligned}
&\e\frac{2s-1}{s^2}\iint_{\mathcal{C}}|\nabla_{x} (v^{s})|^{2}\,dx\,dy+\frac{2s-1}{s^2}\iint_{\mathcal{C}} |(v^{s})_{y}|^{2}\,dx\,dy+\int_{\Omega}|u^{s}|^{2}\,dx=\int_{\Omega}|u|^{p-1+2s}\,dx\label{eq.29}\\
&\leq C^{p-1}\e^{1/2}\int_{\Omega}u^{2s}\,dx.
\end{aligned}
\end{equation}
Using Lemma \ref{thm:traces} (inequality \eqref{eq.20}) and Theorem \ref{thm:trace and H Omega} we have
\[
\e\iint_{\mathcal{C}}|\nabla_{x} (v^{s})|^{2}\,dx\,dy+\iint_{\mathcal{C}} |(v^{s})_{y}|^{2}\,dx\,dy\geq
\|(-\e\Delta_{N})^{1/4} u^{s}\|_{L^{2}(\Omega)}^2\geq
C_{1}\e^{1/2}[u^{s}]^{2}_{H^{1/2}(\Omega)}.
\]
Plugging this estimate into \eqref{eq.29} and noticing that $s^{2}/(2s-1)\leq s$ for $s\geq1$, we get
\[
[u^{s}]^{2}_{H^{1/2}(\Omega)}\leq C_{2}s\int_{\Omega}u^{2s}\,dx.
\]
Thus
\begin{equation}
\|u^{s}\|^{2}_{H^{1/2}(\Omega)}=\|u^{s}\|^{2}_{L^{2}(\Omega)}+[u^{s}]^{2}_{H^{1/2}(\Omega)}\leq C_{3}s\int_{\Omega}u^{2s}\,dx \label{eq.28}
\end{equation}
where the constants $C_{i}$, $i=1,2,3$ depend only on $n$ and $\Omega$. Hence, by applying the fractional Sobolev embedding $H^{1/2}(\Omega)\hookrightarrow L^{\nu}(\Omega)$ in \eqref{eq.28}, where
\[
\nu=\frac{2n}{n-1},
\]
we find
\begin{equation}
\left(\int_{\Omega}u^{s\nu}\,dx\right)^{2/\nu}\leq C_{4} s \int_{\Omega}u^{2s}\,dx\label{eq.31},
\end{equation}
for some constant $C_{4}=C_{4}(n)$.
At this point we are ready to proceed as in \cite[pp.~21--22]{Lin-Ni-Takagi}. Set
\[
r_{1}=p, \quad r_{j+1}=2^{-1}\nu r_{j},
\]
so that
\[
r_{j}=p\,(2^{-1}\nu)^{j-1},
\]
and put
\[
\alpha_{j}=\int_{\Omega}u^{r_{j}}\,dx, \quad j\geq1.
\]
Then, by \eqref{eq.31},
\[
\alpha_{j+1}\leq(C_{5}r_{j})^{\nu/2}\alpha_{j}^{\nu/2},
\]
where $C_{5}=C_{4}/2$ and, as in \cite[p.~21]{Lin-Ni-Takagi},
\[
\limsup_{j\rightarrow\infty}r_{j}^{-1}\log\alpha_{j}\leq\frac{1}{p}\left[\log\alpha_{1}+\nu^{\ast}(\nu^{\ast}-1)^{-1}\left\{\log (p\,C_{4})+(\nu^{\ast}-1)^{-1}\log\nu^{\ast}\right\}\right]
\]
where $\nu^{\ast}=\nu/2$. Thus
\[
\|u\|_{L^{\infty}(\Omega)}\leq C_{6}\alpha_{1}^{1/p},
\]
for some suitable constant $C_{6}=C_6(n,\Omega)$. By integrating the first equation in \eqref{eq.10} over
 $\mathcal{C}$, we have
\[
\int_{\Omega}u\,dx=\int_{\Omega}u^{p}\,dx.
\]
Hence, by H\"older's inequality,
\[
\int_{\Omega}u^{p}\,dx\leq|\Omega|^{(p-1)/p}\left(\int_{\Omega}u^{p}\,dx\right)^{1/p},
\]
so
\[
\alpha_{1}=\int_{\Omega}u^{p}\,dx\leq |\Omega|,
\]
namely
\[
\|u\|_{L^{\infty}(\Omega)}\leq C_{6}\,|\Omega|^{1/p},
\]
and the proof is complete.
\end{proof}

%%%%%%%%%%%%%%%%%%%%%%%%%%%%%%%%%%%%%%%%%%%%%%%%%%%%%%
\subsection{Nonexistence for large $\e$}
%%%%%%%%%%%%%%%%%%%%%%%%%%%%%%%%%%%%%%%%%%%%%%%%%%%%%%

As a consequence of the boundedness result contained in Theorem \ref{Unifboundtheo}
and by following  {ideas contained in} \cite{MR849484}
we are able now to show that $u\equiv1$ is actually the only positive solution to \eqref{eq.8} for large $\e$.

\begin{theorem}\label{existconst}
There exists $\e^{\ast}>0$ such that if $\e>\e^{\ast}$, then $u\equiv1$ is the only positive solution to \eqref{eq.8}.
\end{theorem}

\begin{proof}
Let $u$ be a positive solution to \eqref{eq.8} and write $u=\phi+u_{\Omega}$, where $u_\Omega$
is as in \eqref{eq:mean}, so that
\begin{equation}\label{eq:tres estrellases}
\int_\Omega\phi\,dx=0.
\end{equation}
Then $\phi$ satisfies the equation (recall that $g(u)=u^p$ when $u>0$)
\[
(-\e\Delta_{N})^{1/2}\phi+\phi-\left(\int_{0}^{1}p(u_{\Omega}+t\phi)^{p-1}\,dt\right)\phi=u_{\Omega}^p-u_{\Omega}.
\]
Let $v^{\phi}=E^{\e}(\phi)$ be the $\e$--Neumann extension of $\phi$, which satisfies the extension problem
\begin{equation}\label{eq.phi}
\begin{cases}
\e\Delta_xv^{\phi}+v^{\phi}_{yy}=0,&\hbox{in}~\mathcal{C},\\
\partial_{\nu} v^{\phi}=0,&\hbox{on}~\partial_{L}\mathcal{C},\\
\displaystyle-\lim_{y\rightarrow0}v^{\phi}_y(\cdot,y)=\left(\int_{0}^{1}p(u_{\Omega}+t\phi)^{p-1}\,dt\right)\phi-\phi+
u_{\Omega}^p-u_{\Omega},&\hbox{on}~\Omega.
\end{cases}
\end{equation}
Taking $v^{\phi}$ as a test function in \eqref{eq.phi} and using \eqref{eq:tres estrellases}
 we find the identity
\[
\e\iint_{\mathcal{C}}|\nabla_{x} v^{\phi}|^{2}\,dx\,dy+\iint_{\mathcal{C}} |v^{\phi}_{y}|^{2}\,dx\,dy+\int_{\Omega}\phi^{2}\,dx
=\int_{\Omega}\left(\int_{0}^{1}p(u_{\Omega}+t\phi)^{p-1}\,dt\right)\phi^{2}\,dx.
\]
Since by Theorem \ref{Unifboundtheo} we have
\[
\sup_{\Omega} u\leq C
\]
where $C$ is a constant not depending on $\e$, we find
\[
\e\iint_{\mathcal{C}}|\nabla_{x} v^{\phi}|^{2}\,dx\,dy+\iint_{\mathcal{C}} |v^{\phi}_{y}|^{2}\,dx\,dy+\int_{\Omega}\phi^{2}\,dx
\leq pC^{p-1}\int_{\Omega}\phi^{2}\,dx.
\]
Thus inequality \eqref{eq.20} yields
\[
\|(-\e\Delta_{N})^{1/4} \phi\|^{2}_{L^{2}(\Omega)}+\int_{\Omega}\phi^{2}\,dx\leq p C^{p-1}\int_{\Omega}\phi^{2}\,dx,
\]
which in turn implies, by Theorem \ref{thm:trace and H Omega}, that for some constant $C_{1}>0$,
\begin{equation}
\e^{1/2}C_{1}[\phi]^{2}_{H^{1/2}(\Omega)}+\int_{\Omega}\phi^{2}\,dx\leq pC^{p-1}\int_{\Omega}\phi^{2}\,dx.\label{eq.26}
\end{equation}
Now we recall the \emph{fractional Poincar\'{e} inequality} (see \cite{Adams})
 which says that there is a constant $C_{2}>0$ such that for all
$\psi\in H^{1/2}(\Omega)$ one has
\[
C_{2}\|\psi-\psi_{\Omega}\|_{L^{2}(\Omega)}\leq[\psi]_{H^{1/2}(\Omega)}.
\]
Then applying such inequality to $\psi=\phi$, by recalling \eqref{eq:tres estrellases} and inserting it into \eqref{eq.26} we finally find
\[
(C_{1}C_{2}\e^{1/2}+1)\int_{\Omega}\phi^{2}\,dx\leq p\, C^{p-1}\int_{\Omega}\phi^{2}\,dx,
\]
which is impossible if $\e>\left[(p\, C^{p-1}-1)/C_1C_{2})_{+}\right]^{2}=:\e^{\ast}$ and $\phi\not\equiv 0$. Then for $\e>\e^{\ast}$ we must have $\phi\equiv0$, namely $u=u_\Omega$ and \eqref{eq.8}
implies $u\equiv1$.
\end{proof}
\noindent\textbf{Acknowledgements.}
This research was motivated from discussions between the second author and Christian Kuehn.
We thank Laurent Saloff-Coste and Jiaping Wang for
very interesting discussions regarding the Neumann heat kernel,
 {and to Luis Caffarelli for pleasant conversations about this work.
We are also grateful to Benedetta Pellacci for pointing out a computational mistake
in an earlier version of this paper, as well as to the referee for very useful detailed remarks that helped us
to improve the presentation of the results.}
The first author is grateful to the
Dipartimento di Ingegneria at Universit\`a degli Studi di Napoli ``Parthenope'' for their kind hospitality during
several visits. The authors have been partially supported by MTM2011-28149-C02-01 from Spanish Government and the Gruppo
Nazionale per l'Analisi Matematica, la Probabilit\`{a} e le loro
Applicazioni (GNAMPA) of the Istituto Nazionale di Alta Matematica
(INdAM), Italy.

\bibliographystyle{siam}\small
\bibliography{FractNeumbibJune}

\begin{thebibliography}{10}

\bibitem{Adams}
{\sc R.~A. Adams}, {\em Sobolev Spaces}, Academic Press, New York-London, 1975.
\newblock Pure and Applied Mathematics, Vol. 65.

\bibitem{Alves-Oliva}
{\sc M.~O. Alves and S.~M. Oliva}, {\em An extension problem related to the
  square root of the laplacian with neumann boundary condition}, Electron. J.
  Differential Equations, 2014 (2014), p.~18.

\bibitem{AMBRRAB}
{\sc A.~Ambrosetti and P.~H. Rabinowitz}, {\em Dual variational methods in
  critical point theory and applications}, J. Functional Analysis, 14 (1973),
  pp.~349--381.

\bibitem{Cabre-Sola}
{\sc X.~Cabr{\'e} and J.~Sol{\`a}-Morales}, {\em Layer solutions in a
  half-space for boundary reactions}, Comm. Pure Appl. Math., 58 (2005),
  pp.~1678--1732.

\bibitem{Cabre-Tan}
{\sc X.~Cabr\'e and J.~Tan}, {\em Positive solutions of nonlinear problems
  involving the square root of the laplacian}, Adv. Math., 224 (2010),
  pp.~2052--2093.

\bibitem{Caffree}
{\sc L.~Caffarelli}, {\em Free boundary problems for fractional powers of the
  {L}aplacian}, in A great mathematician of the nineteenth century. {P}apers in
  honor of {E}ugenio {B}eltrami (1835--1900) ({I}talian), vol.~39 of Ist.
  Lombardo Accad. Sci. Lett. Incontr. Studio, LED--Ed. Univ. Lett. Econ.
  Diritto, Milan, 2007, pp.~273--286.

\bibitem{Caffarelli-Silvestre}
{\sc L.~Caffarelli and L.~Silvestre}, {\em An extension problem related to the
  fractional {L}aplacian}, Comm. Partial Differential Equations, 32 (2007),
  pp.~1245--1260.

\bibitem{Caffarelli-Salsa-Silvestre}
{\sc L.~A. Caffarelli, S.~Salsa, and L.~Silvestre}, {\em Regularity estimates
  for the solution and the free boundary of the obstacle problem for the
  fractional {L}aplacian}, Invent. Math., 171 (2008), pp.~425--461.

\bibitem{Caffgeo}
{\sc L.~A. Caffarelli and A.~Vasseur}, {\em Drift diffusion equations with
  fractional diffusion and the quasi-geostrophic equation}, Ann. of Math. (2),
  171 (2010), pp.~1903--1930.

\bibitem{Campanato}
{\sc S.~Campanato}, {\em Propriet\`a di h\"olderianit\`a di alcune classi di
  funzioni}, Ann. Scuola Norm. Sup. Pisa (3), 17 (1963), pp.~175--188.

\bibitem{Escudero}
{\sc C.~Escudero}, {\em The fractional keller--segel model}, Nonlinearity, 19
  (2006), pp.~2909--2918.

\bibitem{Evans}
{\sc L.~C. Evans}, {\em Partial Differential Equations}, vol.~19 of Graduate
  Studies in Mathematics, American Mathematical Society, Providence, RI,
  second~ed., 2010.

\bibitem{Gale-Miana-Stinga}
{\sc J.~E. Gal\'e, P.~J. Miana, and P.~R. Stinga}, {\em Extension problem and
  fractional operators: semigroups and wave equations}, J. Evol. Equ., 13
  (2013), pp.~343--368.

\bibitem{Gidas-Spruck}
{\sc B.~Gidas and J.~Spruck}, {\em A priori bounds for positive solutions of
  nonlinear elliptic equations}, Comm. Partial Differential Equations, 6
  (1981), pp.~883--901.

\bibitem{Gilbarg-Trudinger}
{\sc D.~Gilbarg and N.~S. Trudinger}, {\em Elliptic Partial Differential
  Equations of Second Order}, Classics in Mathematics, Springer-Verlag, Berlin,
  2001.
\newblock Reprint of the 1998 edition.

\bibitem{Gyrya-Saloff-Coste}
{\sc P.~Gyrya and L.~Saloff-Coste}, {\em Neumann and dirichlet heat kernels in
  inner uniform domains}, Ast\'erisque, 336 (2011), pp.~viii+144.

\bibitem{Harboure}
{\sc E.~Harboure}, {\em Spaces of smooth functions}, in Advanced courses of
  mathematical analysis III, Hackensack, NJ, 2008, World Sci. Publ.,
  pp.~67--85.

\bibitem{Hu}
{\sc B.~Hu}, {\em Nonexistence of a positive solution of the laplace equation
  with a nonlinear boundary condition}, Differential Integral Equations, 7
  (1994), pp.~301--313.

\bibitem{Imbert-Mellet}
{\sc C.~Imbert and A.~Mellet}, {\em Existence of solutions for a higher order
  non-local equation appearing in crack dynamics}, Nonlinearity, 24 (2011),
  pp.~3487--3514.

\bibitem{YanYanLi}
{\sc Y.~Li and L.~Zhang}, {\em Liouville-type theorems and {H}arnack-type
  inequalities for semilinear elliptic equations}, J. Anal. Math., 90 (2003),
  pp.~27--87.

\bibitem{Lin-Ni-Takagi}
{\sc C.-S. Lin, W.-M. Ni, and I.~Takagi}, {\em Large amplitude stationary
  solutions to a chemotaxis system}, J. Differential Equations, 72 (1988),
  pp.~1--27.

\bibitem{PellacciMontef}
{\sc E.~Montefusco, B.~Pellacci, and G.~Verzini}, {\em Fractional diffusion
  with {N}eumann boundary conditions: the logistic equation}, Discrete Contin.
  Dyn. Syst. Ser. B, 18 (2013), pp.~2175--2202.

\bibitem{MR849484}
{\sc W.-M. Ni and I.~Takagi}, {\em On the {N}eumann problem for some semilinear
  elliptic equations and systems of activator-inhibitor type}, Trans. Amer.
  Math. Soc., 297 (1986), pp.~351--368.

\bibitem{NiTakagi2}
\leavevmode\vrule height 2pt depth -1.6pt width 23pt, {\em On the shape of
  least-energy solutions to a semilinear {N}eumann problem}, Comm. Pure Appl.
  Math., 44 (1991), pp.~819--851.

\bibitem{NiTakagi3}
\leavevmode\vrule height 2pt depth -1.6pt width 23pt, {\em Locating the peaks
  of least-energy solutions to a semilinear {N}eumann problem}, Duke Math. J.,
  70 (1993), pp.~247--281.

\bibitem{Saloff-Coste}
{\sc L.~Saloff-Coste}, {\em The heat kernel and its estimates}, vol.~57 of Adv.
  Stud. Pure Math., Math. Soc. Japan, 2010, pp.~405--436.

\bibitem{Silvestre-thesis}
{\sc L.~Silvestre}, {\em Regularity of the obstacle problem for a fractional
  power of the Laplace operator}, PhD thesis, The University of Texas at
  Austin, USA, 2005.

\bibitem{Silv1}
\leavevmode\vrule height 2pt depth -1.6pt width 23pt, {\em H\"older estimates
  for solutions of integro-differential equations like the fractional
  {L}aplace}, Indiana Univ. Math. J., 55 (2006), pp.~1155--1174.

\bibitem{Silv2}
\leavevmode\vrule height 2pt depth -1.6pt width 23pt, {\em Regularity of the
  obstacle problem for a fractional power of the {L}aplace operator}, Comm.
  Pure Appl. Math., 60 (2007), pp.~67--112.

\bibitem{Stein-Singular}
{\sc E.~M. Stein}, {\em Singular Integrals and Differentiability Properties of
  Functions}, Princeton Univ. Press, Princeton, NJ, 1970.

\bibitem{Stinga}
{\sc P.~R. Stinga}, {\em Fractional powers of second order partial differential
  operators: extension problem and regularity theory}, PhD thesis, Universidad
  Aut\'onoma de Madrid, Spain, 2010.

\bibitem{Stinga-Torrea}
{\sc P.~R. Stinga and J.~L. Torrea}, {\em Extension problem and {H}arnack's
  inequality for some fractional operators}, Comm. Partial Differential
  Equations, 35 (2010), pp.~2092--2122.

\bibitem{Wang-Yan-Preprint}
{\sc F.-Y. Wang and L.~Yan}, {\em Gradient estimate on the neumann semigroup
  and applications}, arXiv:1009.1965v2,  (2010), p.~12.

\bibitem{Wang-Yan-Proceedings}
\leavevmode\vrule height 2pt depth -1.6pt width 23pt, {\em Gradient estimate on
  convex domains and applications}, Proc. Amer. Math. Soc., 141 (2012),
  pp.~1067--1081.

\bibitem{Wang}
{\sc J.~Wang}, {\em Global heat kernel estimates}, Pacific J. Math., 178
  (1997), pp.~377--398.

\end{thebibliography}
\end{document}